\documentclass[notitlepage,11pt,reqno]{amsart}
\usepackage{amssymb,amsmath,color,tocvsec2}
\usepackage[mathscr]{eucal}
\usepackage[breaklinks=true]{hyperref}

\usepackage{pgf,tikz}
\usetikzlibrary{arrows}
\usepackage[letterpaper]{geometry}
\newcommand{\Z}{\ensuremath{\mathbb Z}}

\newcommand{\dt}{\ensuremath{D_\mathcal{T}}}
\newcommand{\at}{\ensuremath{\alpha_\mathcal{T}}}
\newcommand{\tild}{\ensuremath{\widetilde{D}}}
\renewcommand{\geq}{\ensuremath{\geqslant}}
\renewcommand{\leq}{\ensuremath{\leqslant}}
\newcommand{\tauin}{\boldsymbol{\tau}_{\rm in}}
\newcommand{\tauout}{\boldsymbol{\tau}_{\rm out}}

\DeclareMathOperator{\cof}{cof}
\DeclareMathOperator{\Id}{Id}
\DeclareMathOperator{\adj}{adj}
\DeclareMathOperator{\Toep}{Toep}
\DeclareMathOperator{\diag}{diag}

\newtheorem{theorem}{Theorem}[section]
\newtheorem{lemma}[theorem]{Lemma}
\newtheorem{prop}[theorem]{Proposition}
\newtheorem{cor}[theorem]{Corollary}

\newtheorem{utheorem}{\textrm{\textbf{Theorem}}}

\theoremstyle{definition}
\newtheorem{definition}[theorem]{Definition}
\newtheorem{remark}[theorem]{Remark}
\newtheorem{notation}[theorem]{Definition-Notation}
\newtheorem{example}[theorem]{Example}

\numberwithin{equation}{section}

\begin{document}
\vspace*{-1mm}

%{{{1 Frontmatter
\title[Distance matrices of a tree: two more invariants, and in a unified
framework]{Distance matrices of a tree:\\ two more invariants, and in a
unified framework}

\dedicatory{To Ravindra B.\ Bapat with admiration and thanks, for
introducing us to distance matrices}

\author{Projesh Nath Choudhury}
\address[P.N.~Choudhury]{Discipline of Mathematics, Indian Institute of
Technology Gandhinagar, Palaj, Gandhinagar 382355, India}
\email{\tt projeshnc@iitgn.ac.in}

\author{Apoorva Khare}
\address[A.~Khare]{Department of Mathematics, Indian Institute of
Science, and Analysis and Probability Research Group, Bangalore 560012,
India}
\email{\tt khare@iisc.ac.in}



\begin{abstract}
A classical result of Graham and Pollak [\textit{Bell Sys. Tech. J.}
1971] states that the determinant of the distance matrix $D_T$ of any
tree $T$ depends only on the number of edges of $T$. This and several
other variants of $D_T$ have since been studied -- including a
$q$-version, a multiplicative version, and directed versions of these --
and in all cases, $\det(D_T)$ depends only on the edge-data.

In this paper, we introduce a more general framework for bi-directed
weighted trees that has not been studied to date; our work is significant
for three reasons. First, our setting strictly generalizes -- and unifies
-- all variants of $D_T$ studied to date (with coefficients in an
arbitrary unital commutative ring) -- including in
[\textit{Bell Sys. Tech. J.} 1971] above, as well as
[\textit{Adv. Math.} 1978],
[\textit{J. Combin. Theory Ser. A} 2006],
[\textit{Adv. Appl. Math.} 2007],
[\textit{Electron. J. Combin.} 2010],
%and [\textit{Linear Algebra Appl.} 2005, 2006, 2009, 2015, 2016]
and others. %(see references)

Second, our results strictly improve on state-of-the-art for every
variant of the distance matrix studied to date, even in the classical
Graham--Pollak case. Here are three results for trees:
(1)~We compute the minors obtained by deleting arbitrary equinumerous
sets of pendant nodes (in fact, more general sub-forests) from the rows
and columns of $D_T$, and show these minors depend only on the edge-data
and not the tree-structure. 
(2)~We compute a second function of the distance matrix $D_T$: the sum of
all its cofactors, termed $\cof(D_T)$. We do so in our general setting
\textit{and} in stronger form, after deleting equinumerous pendant nodes
(and more generally) as above -- and show these quantities also depend
only on the edge-data.
(3)~We compute in closed form the inverse of $D_T$, extending a result of
Graham and Lov\'asz~[\textit{Adv. Math.} 1978] and answering an open
question of Bapat--Lal--Pati [\textit{Linear Algebra Appl.} 2006] in
greater generality.

Third, a new technique is to crucially use commutative algebra arguments
-- specifically, Zariski density -- which to our knowledge are hitherto
unused for such matrices/invariants, but are richly rewarding. We also
explain why our setting is ``most general'', in that for more general
edgeweights, $\det(D_T), \cof(D_T)$ depend on the tree structure. In a
sense, this completes the study of the invariants $\det(D_T), \cof(D_T)$
for distance matrices of trees $T$ with edge-data in a commutative ring.

Our proofs use novel results for arbitrary bi-directed strongly connected
graphs $G$: we prove a multiplicative analogue of an additive result by
Graham--Hoffman--Hosoya [\textit{J.~Graph Theory} 1977], as well as a
novel $q$-version thereof. In particular, we provide closed-form
expressions for $\det(D_G)$, $\cof(D_G)$, and $D_G^{-1}$ in terms of
their strong blocks. We then show how this subsumes the classical 1977
result, and provide sample applications to adding pendant trees and to
cycle-clique graphs (including cactus/polycyclic graphs and hypertrees),
subsuming variants in the literature. The final section introduces and
computes a third -- and novel -- invariant for trees, as well as a
parallel Graham--Hoffman--Hosoya type result for our ``most general''
distance matrix $D_T$.
\end{abstract}

\date{\today}
\maketitle
%}}}

%%%%%%%%%%%%%%%%%%%%%%%%%%%%%%%%%%%%
\vspace*{-6mm}
%%%%%%%%%%%%%%%%%%%%%%%%%%%%%%%%%%%%

\settocdepth{section}
\tableofcontents
\newpage

\textit{We work over an arbitrary unital commutative ground ring $R$,
unless otherwise specified. For a fixed integer $n \geq 1$, we define
$[n] := \{ 1, \dots, n \}$, ${\bf e} = {\bf e}(n) := (1,\dots,1)^T \in
R^n$, and $J_{n \times n} := {\bf e} {\bf e}^T$. The standard basis of
$R^n$ will be denoted by ${\bf e}_1, \dots, {\bf e}_n$. Also, given a
matrix $A = (a_{ij}) \in R^{n \times n}$ with cofactors $c_{ij} =
(-1)^{i+j} \det A_{ij}$, its adjugate matrix is $\adj(A) :=
(c_{ji})_{i,j=1}^n$.}

\textit{Recall that a tree is a finite connected graph $T = (V,E)$ with
$|E| = |V|-1$, or equivalently, with a unique path between any two
vertices. We write $i \sim j$ to mean that $i \neq j$ and $i,j$ are
adjacent in $T$: $\{ i,j \} \in E$. Given a pendant node $i \in V$, we
denote the unique node adjacent to it by $p(i)$.}

%{{{1 Section 1 - General framework and main results
\section{General framework and main results}

This paper contributes to the study of matrices associated to a graph
$G$ -- see e.g.~\cite{Bap,Br,BrCv} for a rich history and detailed
information. Specifically, we work with distance matrices.
Given an unweighted, undirected tree $T$ with node set $V$, and nodes
$v,w \in V$, let $d(v,w)$ denote the integer length of the unique path
from $v$ to $w$; thus $d(v,v) := 0$. Now define the \textit{distance
matrix} $D_T$ to be the $V \times V$ matrix with $(v,w)$ entry $d(v,w)$.
Such matrices and their variants have connections to communication
networks, network flow algorithms, quantum chemistry and molecular
stability, and graph embeddings. For more information, see
e.g.~\cite{GHH,YY2} and the references therein.

We begin with a well-known and striking result from fifty years ago, by
Graham and Pollak in~\cite{Graham-Pollak}. Namely, if $D_T$ denotes the
$n \times n$ path-distance matrix (with entries in $\Z^{\geq 0}$) for a
tree $T$ with node set $[n]$ and edge-set $E$, then $\det(D_T)$ does not
depend on the tree-structure of $T$:
\begin{equation}\label{Egraham-pollak}
\det(D_T) = (-1)^{|E|} |E| 2^{|E|-1},
\end{equation}

This has since been extended in several ways. For instance,
Bapat--Kirkland--Neumann~\cite{BKN} and Yan--Yeh~\cite{YY2} consider
weighted and undirected trees, in which case the distance $d(v,w)$ is
taken as the sum of the edgeweights along the unique path from $v$ to
$w$. This further extends to using the unique directed path if one
considers the more general case of a directed weighted tree with directed
edgeweights $a_e, a'_e$ between the vertices of $e$. Even in this
generality, a result similar to~\eqref{Egraham-pollak} holds:

\begin{theorem}[Bapat--Lal--Pati~\cite{BLP2}]\label{Tblp}
Given a tree $T = (V,E)$ with edgeweights $\{ a_e, a'_e : e \in E \},$
\[
\det (D_T) = (-1)^{|E|} \sum_{e \in E} a_e a'_e \prod_{f \in E, \ f \neq
e} (a_f + a'_f).
\]
\end{theorem}

Note that $\det(D_T)$ is ``doubly symmetric'' in its edgeweights, in that
(a)~it is independent of the tree structure $e \mapsto \{ a_e, a'_e \}$,
in the flavor of~\eqref{Egraham-pollak}; and
(b)~it is also independent of the ``orientation assignment'' $(a_e,
a'_e)$ or $(a'_e, a_e)$ for a pair of oppositely directed edges.

As may be expected, these results soon led to $q$-versions; now one
replaces $a_e$ by $[a_e] := \frac{q^{a_e} - 1}{q-1}$ (and also for
$a'_e$) either for a parameter $q$ (by Yan--Yeh~\cite{YY2}) or for $q$
real, e.g. $q \neq 1$ (by Bapat--Lal--Pati~\cite{BLP1}) or $q \neq \pm 1$
(by Bapat--Rekhi~\cite{BR}). These works considered the undirected case,
which was subsequently extended to bi-directed trees:

\begin{theorem}[Li--Su--Zhang~\cite{LSZ}]\label{Tlsz}
Given a tree $T = (V,E)$ with edgeweights $\{ a_e, a'_e : e \in E \}$,
let its $q$-distance matrix $D_q(T)$ have $(v,w)$ entry $[d(v,w)] :=
\frac{q^{d(v,w)} - 1}{q-1}$ if $q \neq 1$, and $d(v,w)$ otherwise, where
the original distance matrix $(d(v,w))_{v,w \in V}$ is as in
Theorem~\ref{Tblp}. Then
\[
\det (D_q(T)) = (-1)^{|E|} \sum_{e \in E} [a_e] [a'_e] \prod_{f \in E, \
f \neq e} [a_f + a'_f].
\]
\end{theorem}

\noindent Notice again the doubly symmetric formula; also, this
generalizes Theorem~\ref{Tblp}, hence all other ``additive'' variants
mentioned above.

In a different vein, Bapat--Lal--Pati~\cite{BLP1},
Bapat--Rekhi~\cite{BR}, and Yan--Yeh~\cite{YY2} studied the
``$q$-exponential distance matrix'' $(q^{d(v,w)})_{v,w \in V}$. This was
extended by Bapat--Sivasubramanian~\cite{BS2} to arbitrary multiplicative
edgeweights $\{ m_e, m'_e : e \in E \}$, and separately by
Zhou--Ding~\cite{ZD1}:

\begin{theorem}\label{Tmult}
Given a directed tree with weighted edge-data $\{ m_e, m'_e : e \in E
\}$, let $D^*_T$ denote the $V \times V$ matrix with diagonal entries $1$,
and the $(v,w)$-entry (for $v \neq w$) the product of the multiplicative
edgeweights along the unique directed path $: v \to w$. Then
$\displaystyle \det(D^*_T) = \prod_{e \in E} (1 - m_e m'_e)$.
\end{theorem}

\subsection{Novel, general framework}

In this paper we introduce a more general class of weighted trees which
strictly encompass all of the variants studied to date; and for each such
tree (including in the aforementioned settings), we will prove the above
independence result, but in a stronger form.
More precisely, we work with trees with edgeweights in a unital
commutative ring $R$, and so our results hold for all such $R$. In the
most general such version, each edge $\{ i, j \}$ is also
\textit{bi-directed}, and the weights are pairs of labels. Thus, each
edge $\{ i, j \}$ comes with two pairs of elements
\begin{equation}\label{Edirected}
(a_{i \to j}, m_{i \to j}) \qquad \text{and} \qquad
(a_{j \to i}, m_{j \to i}),
\end{equation}
where $a$ and $m$ are to be thought of as ``additive'' and
``multiplicative'', respectively.\medskip

\begin{notation}\label{Dnot}
We work with a tree $T = (V,E)$ where $V = [n]$, and over an arbitrary
unital commutative ring $R$. In the sequel, we will omit the arrows
in~\eqref{Edirected} and merely write $a_{ij}, m_{ij}$ for vertices $i, j
\in V = [n]$. The corresponding \textbf{tree-data} or set of
\textbf{edgeweights} is denoted by
\begin{equation}
\mathcal{T} = \mathcal{T}(T) := \{ (i,j; a_{ij}, m_{ij}; a_{ji}, m_{ji})
\, : \, i \sim j, \  i < j \}.
\end{equation}

As we explain below, all variants in the literature (and in this paper)
involve using $a_{ij} = a_{ji}\ \forall i,j$. Thus, for an edge $e = \{
i, j \}$ we will denote symmetric functions in $m_{ij}, m_{ji}$ using the
symbols $m_e, m'_e$. We will then also write $a_e = a_{ij} = a_{ji}$, and
call the \textit{triple} $(a_e, m_e, m'_e)$ as the \textbf{edgeweight}
for $e \in E$. (As we see below, this is a mild abuse of notation for
$(a_e, \{ m_e, m'_e \})$.)

With this notation, the \textbf{directed distance matrix} associated to
$\mathcal{T}$ is the matrix $\dt$, with $(i,j)$ entry $w_{i \to j}$
defined as follows: let the unique directed path from $i$ to $j$ be given
by
\[
i =: i_0 \ \longrightarrow \ i_1 \ \longrightarrow \ \cdots \
\longrightarrow \ i_k := j, \qquad k \geq 0.
\]
Now define $\dt := (w_{i \to j})_{i,j=1}^n$, where
$w_{i \to i} := 0$; and for $i \neq j$,
\begin{equation}\label{Eedgeweight}
w_{i \to j} := a_{i_0 i_1} (m_{i_0 i_1} - 1) + a_{i_1 i_2} (m_{i_0
i_1} m_{i_1 i_2} - m_{i_0 i_1}) + \cdots
= \sum_{l=0}^{k-1} a_{i_l i_{l+1}} (m_{i_l i_{l+1}} - 1)
\prod_{u=0}^{l-1} m_{i_u i_{u+1}}.
\end{equation}
%where the empty product $\prod_{u=0}^{-1}$ is defined to be one.
\end{notation}

\begin{example}[The tree on 3 nodes]
For the tree in Figure~\ref{Fig1}, the corresponding matrix $\dt$ is
\begin{equation}\label{Edtp3}
\begin{pmatrix}
0 & a_{12} (m_{12}-1) & a_{12} (m_{12} - 1) + a_{23} m_{12} (m_{23} - 1) \\
a_{12} (m_{21} - 1) & 0 & a_{23} (m_{23} - 1) \\
a_{23} (m_{32} - 1) + a_{12} m_{32} (m_{21} - 1) & a_{23} (m_{32} - 1) & 0
\end{pmatrix}.
\end{equation}

\begin{figure}[ht]
\begin{tikzpicture}[line cap=round,line join=round,>=triangle 45,x=1.0cm,y=1.0cm]
\draw (1,1)-- (7,1);
\draw (1.5,1.7) node[anchor=north west] {$ (a_{12}, m_{12}) $};
\draw (1.5,1) node[anchor=north west] {$ (a_{12}, m_{21}) $};
\draw (4.5,1.7) node[anchor=north west] {$ (a_{23}, m_{23}) $};
\draw (4.5,1) node[anchor=north west] {$ (a_{23}, m_{32}) $};
\begin{scriptsize}
\fill (1,1) circle (2.5pt);
\draw (1,1.3) node {$1$};
\fill (4,1) circle (2.5pt);
\draw (4,1.3) node {$2$};
\fill (7,1) circle (2.5pt);
\draw (7,1.3) node {$3$};
\end{scriptsize}
\end{tikzpicture}
\caption{The tree $T = P_3$, with general edge-data $\mathcal{T}$.}
\label{Fig1}
\end{figure}
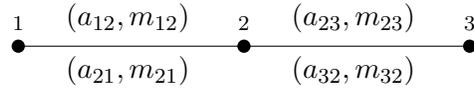
\end{example}

Notice (e.g.~in the above example) that $\dt$ need not be symmetric in
our model. Indeed, this is the case in several previous papers, see
e.g.~\cite{BLP1,BLP2,BS2,GHH,cactoid,ZD1,ZD2}.

\begin{remark}\label{Rspecialize}
To our knowledge, the above ``additive-multiplicative'' setting
encompasses all previous variants in the literature. For example, the
setting of Theorem~\ref{Tlsz} with $q \neq 1$ involves
setting all additive edgeweights to be $1/(q-1)$, and
$m_e = q^{a_e}$, $m'_e = q^{a'_e}$
(with a mild abuse of notation). Thus, a formula for $\det(\dt)$ in our
setting would extend the one in Theorem~\ref{Tlsz}, hence all prior
variants. Note, such a formula would naively work only for $q^{a_e}$ with
$a_e \in \mathbb{Z}$, and so the $q \to 1$ case provides formulas for
$\det(\dt)$ that work only for $a_e \in \mathbb{Z}$. However, we show
that the same formula will work for $a_e$ in any unital commutative ring,
using the power of Zariski density -- a novel technique in this context,
which we introduce and explain in multiple proofs. Thus, our general
framework subsumes all ``additive'' variants of $\dt$ considered in the
literature.

Similarly, the multiplicative setting of Theorem~\ref{Tmult} is recovered
by setting $a_e = 1$. Then the ``multiplicative distance matrix'' is
given by $\dt + J$, where $J$ is the all-ones matrix.
\end{remark}

Our first main result implies as a special case that all of the above
distance matrices have determinants independent of the tree structure,
and depending only on the edge-data. As the preceding remark shows, we
will require a formula for $\det(D + xJ)$, where $x$ is a scalar. This
necessitates recalling a notion studied by
Graham--Hoffman--Hosoya~\cite{GHH}:

\begin{definition}
Given a square matrix $A$, its \textbf{cofactor-sum} $\cof(A)$ is defined
to be the sum of all cofactors of $A$, namely, the sum over all $i,j \in
[n]$ of $(-1)^{i+j} \det(A_{i|j})$. Here, $A_{i|j}$ is the submatrix of
$A$ obtained by deleting the $i$th row and $j$th column.
\end{definition}

We immediately record -- and use below, occasionally without further
reference -- the following straightforward facts from linear algebra. See
e.g.~\cite{Bap,GHH} for a proof of these or of close variants:

\begin{lemma}\label{Ldetcof}
Let $R$ be any unital commutative ring, $A \in R^{n \times n}$ any square
matrix (for $n \geq 1$), and $x$ an indeterminate that commutes with $R$.
Then $\det(A + xJ) = \det(A) + x \cof(A)$. Moreover,
\[
\cof(A) = {\bf e}^T \adj(A) {\bf e} = \cof(A + xJ),
\]
and $\adj(A + xJ) {\bf e}$ does not depend on $x$.
\end{lemma}

The quantity $\cof(\dt)$ was first studied by Graham, Hoffman, and Hosoya
in~\cite{GHH} for arbitrary graphs $G$, in the special case of additive
edgeweights $a_e, a'_e$.

\subsection{Main results and novel features}

In this section, we describe the various novel features of the paper.
In a nutshell:
(a)~We prove formulas in our general setting for trees, which specialize
to novel identities for every single variant -- even in the original
Graham--Pollak case of integer-valued unweighted distance matrices.
To our knowledge, such a unification has not been achieved to date.
(b)~As a consequence, the invariance of $\det(\dt), \cof(\dt)$ holds in
our general setting, implying the same in all previous cases.
(c)~We compute the inverse of $\dt$ in our general setting, resolving an
open question of Bapat et al~\cite{BLP1} in greater generality (and
recovering all such formulas for trees in the literature, over arbitrary
unital commutative rings).
(d)~These results use in part, strengthenings of state-of-the-art for
\textit{arbitrary} strongly connected graphs, which will be the subject of
Section~\ref{Sghhmult}.
(e)~In Section~\ref{Sthird} we introduce a novel, edge-multiplicative
invariant for trees, which we term $\kappa(\dt)$. We then formulate and
prove identities relating $\det, \cof$, and $\kappa$.

A satisfying feature: further generalizing our setting leads to
both $\det(\dt), \cof(\dt)$ depending on the tree structure. Thus, in a
sense, our framework is ``\textit{most general}''; see
Example~\ref{Ecounter}. 

Before proceeding to the results, we also stress on a \textbf{novel
technique} that we adopt from commutative algebra: \textit{Zariski
density}. This is immensely rewarding and helps bypass various
artificial technical restrictions; it also can be applied more broadly;
and to our knowledge, it has not been used in this context previously.
See also Remark~\ref{Rzariski} for additional details.\medskip

We next state our main results: two for trees, and one for general graphs
(which has a host of applications that strengthen the existing results in
the literature -- see the next section). Our first result computes in
closed-form both $\det(\dt)$ and $\cof(\dt)$ for
``additive-multiplicative edge-data'' for trees in our general setting
above. In fact, we will compute these quantities for the submatrices of
$\dt$ corresponding to removing equal-sized sets of pendant nodes (and
more generally):

\begin{utheorem}\label{Tmaster}
Suppose a tree $T = (V = [n], E)$ is equipped with edge-data $\mathcal{T}
= \{ (a_{ij} = a_{ji}, m_{ij}, m_{ji}) : \{ i, j \} \in E \}$ as above.
(We write $(a_e, m_e, m'_e)$ for the weights for edges $e \in E$.) 
Let $I,J' \subset V$ satisfy:
(a)~$|I| = |J'| \leq n-3$;
(b)~$|I \cup J'| \leq n-1$;
(c)~$T \setminus I,\ T \setminus J',\ T \setminus (I \cap J')$ are
connected.
Now let $E_\circ := E_{(I \cap J')^c}$ denote the edges in $E$ not among
the common edges adjacent to $I \cap J'$.

As an additional notation, given a $V \times V$ matrix $D$, let
$D_{I|J'}$ denote the submatrix formed by removing the rows and columns
labelled by $I,J'$ respectively. Then $\det (\dt + xJ)_{I|J'}$ depends on
the edge-data but not on the tree structure:
\begin{align}
& \det (\dt + xJ)_{I|J'} \label{Emult-more}\\
& = \begin{cases}
\displaystyle \prod_{e \in E_\circ} (a_e (1 - m_e m'_e)) \left[ x +
\sum_{e \in E_\circ} \frac{(a_e - x) (m_e - 1)(m'_e - 1)}{m_e m'_e - 1}
\right],\notag \quad & \text{if } |I \Delta J'| = 0,\\
\displaystyle \prod_{e \in E_\circ \setminus \{ (p(i_0), i_0), (j_0,
p(j_0)) \}} (a_e (m_e m'_e - 1)) \cdot a_{(p(i_0), i_0)} (a_{(j_0,
p(j_0))} - x) (m_{(p(i_0), i_0)} - 1) (m_{(j_0, p(j_0))} - 1),
\hspace*{-2.6cm}&\vspace*{-3mm} \\
& \text{if } |I \Delta J'| = 2,\\
0, &\text{if } |I \Delta J'| > 2,
\end{cases}
\end{align}
where the denominators (for $I = J'$) are placeholders that cancel with a
factor in $\prod_e (1 - m_e m'_e)$. We also assume that if $|I \Delta J'|
= 2$, then the nodes $i_0, j_0$ are given by $I \setminus J' = \{ i_0 \},
\ J' \setminus I = \{ j_0 \}$.
\end{utheorem}

Theorem~\ref{Tmaster} says more precisely that $\det (\dt + xJ)_{I|J'}$
depends only on the edge-data of the edges in $I \setminus J'$, $J'
\setminus I$, $I \cap J'$, and $E \setminus (I \cap J')$. 
A curious feature is that $\cof(\cdot)$ for $|I \Delta J'| = 2$ is
asymmetric in the additive edge-data for $i_0, j_0$.
Also, the (possibly non-optimal) choice of notation for the index set $J'
\subset [n]$ is to avoid conflict with the all-ones matrix $J$.

We now make several clarifying remarks:

\begin{remark}\label{Rkaushal}
Equation~\eqref{Emult-more} computes principal as well as non-principal
minors of $\dt$. The case when $I \neq J'$ (and both are non-singleton)
of non-principal minors reveals new information about $\dt$ even in the
original setting of Graham--Pollak, hence for every other variant
studied.
\end{remark}

\begin{remark}
Theorem~\ref{Tmaster} includes the originally sought-for case of $I = J'
= \emptyset$:
\begin{align}\label{Emaster}
\begin{aligned}
\det(\dt) = &\ \prod_{e \in E} (a_e (1 - m_e m'_e)) \sum_{e \in E}
\frac{a_e(m_e - 1)(m'_e - 1)}{m_e m'_e - 1},\\
\cof(\dt) = &\ \prod_{e \in E} (a_e (1 - m_e m'_e)) \left[ 1 + \sum_{e
\in E} \frac{(a_e - 1) (m_e - 1)(m'_e - 1)}{m_e m'_e - 1} \right].
\end{aligned}
\end{align}
Note, the invariant $\det(\dt)$ depends only on $(a_e, \{ m_e, m'_e
\})_{e \in E}$, while $\cof(\dt)$ depends on even less: on $\{ a_e : e
\in E \}$ and $\{ \{ m_e, m'_e \} : e \in E \}$ -- and so this holds for
every variant studied to date.
\end{remark}

\begin{remark}
By Theorem~\ref{Tlsz}, the formulas for $\det(\cdot), \cof(\cdot)$ for
the classical distance matrix follow from from their $q$-versions. We
observe in Remark~\ref{Rmiddle} that these $q$-versions follow themselves
from the multiplicative versions where $a_e = a'_e = 1/(q-1)$, but using
$\dt = D^*_T - J$ instead of the multiplicative matrix $D^*_T$.
Thus by Lemma~\ref{Ldetcof}, using $\det(\cdot)$ alone does not help
unify the host of previously studied variants; one crucially has to also
use $\cof(\cdot)$. To our knowledge, such a unification (and its
generalization) had not previously been achieved in the literature.
\end{remark}

\begin{remark}\label{Rmiddle}
As anticipated in Remark~\ref{Rspecialize}: Theorem~\ref{Tblp} (where
$a_e \neq a'_e$) can be deduced from the $a_e = a'_e$
formula~\eqref{Emult-more} as follows: First work over $\mathbb{Q}(q)$
for a parameter $q$, and set all additive edgeweights to be $1/(q-1)$,
and $m_e = q^{a_e}, m'_e = q^{a'_e}$. For integer values of $a_e, a'_e$,
one can use~\eqref{Emult-more} for $I=J=\emptyset$ to deduce
Theorem~\ref{Tlsz}.
Now evaluate at $q=1$ to obtain Theorem~\ref{Tblp}. This works for
weights $a_e, a'_e \in \mathbb{Z}$; the extension to $a_e, a'_e$ in a
commutative ring $R$ follows by Zariski density.
\end{remark}

Our next contribution (after Theorem~\ref{Tmaster}) shows that our
setting is the ``most general'' possible:

\begin{example}\label{Ecounter}
Theorem~\ref{Tmaster} and its special case (by Remark~\ref{Rmiddle})
Theorem~\ref{Tblp} say that in the two cases when
(a)~$m_{ij}, m_{ji}$ need not coincide, while $a_e = a_{ij} = a_{ji}$; and
(b)~$a_{ij}, a_{ji}$ need not coincide, while $m_e = m_{ij} = m_{ji} = q$
and $q \to 1$ -- the terms $\det(\dt), \cof(\dt)$ depend only on the
edge-data, but not on the tree structure.
It is natural to ask if this holds for the remaining variant of weighted
bi-directed trees: where $a_{ij} \neq a_{ji}$ and $m_{ij} = m_{ji} \neq
1$ (or even more generally, $m_{ij} \neq m_{ji}$). The following example
shows that this does not happen in the (perhaps) simplest imaginable
such situation:
suppose $a_{ij}$ and $a_{ji}$ are allowed to be unequal, and set
$m_e = m_{ij} = m_{ji} = q, \ \forall e \in E$.

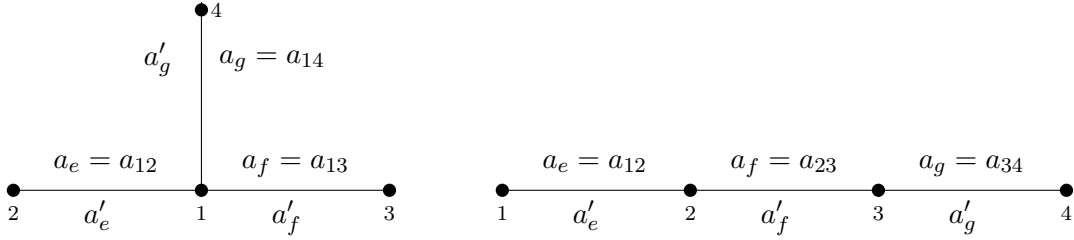
\begin{figure}[ht]
\begin{tikzpicture}[line cap=round,line join=round,>=triangle 45,x=1.0cm,y=1.0cm]
\draw (1,1)-- (6,1);
\draw (1.8,1) node[anchor=north west] {$ a'_e $};
\draw (1.4,1.6) node[anchor=north west] {$ a_e = a_{12} $};
\draw (3.9,1.6) node[anchor=north west] {$ a_f = a_{13} $};
\draw (4.3,1) node[anchor=north west] {$ a'_f $};
\draw (2.6,3.1) node[anchor=north west] {$ a'_g $};
\draw (3.6,3) node[anchor=north west] {$ a_g = a_{14} $};
\begin{scriptsize}
\fill (1,1) circle (2.5pt);
\draw (1,0.7) node {$2$};
\fill (3.5,1) circle (2.5pt);
\draw (3.5,0.7) node {$1$};
\fill (6,1) circle (2.5pt);
\draw (6,0.7) node {$3$};
\draw (3.5,1)-- (3.5,3.5);
\fill (3.5,3.4) circle (2.5pt);
\draw (3.7,3.4) node {$4$};
\end{scriptsize}
\draw (7.5,1)-- (15,1);
\draw (7.9,1.6) node[anchor=north west] {$ a_e = a_{12} $};
\draw (10.4,1.6) node[anchor=north west] {$ a_f = a_{23} $};
\draw (12.9,1.6) node[anchor=north west] {$ a_g = a_{34} $};
\draw (8.3,1) node[anchor=north west] {$ a'_e $};
\draw (10.8,1) node[anchor=north west] {$ a'_f $};
\draw (13.3,1) node[anchor=north west] {$ a'_g $};
\begin{scriptsize}
\fill (7.5,1) circle (2.5pt);
\draw (7.5,0.7) node {$1$};
\fill (10,1) circle (2.5pt);
\draw (10,0.7) node {$2$};
\fill (12.5,1) circle (2.5pt);
\draw (12.5,0.7) node {$3$};
\fill (15,1) circle (2.5pt);
\draw (15,0.7) node {$4$};
\end{scriptsize}

\end{tikzpicture}
\caption{The trees $K_{1,3}$ and $P_4$, with all $m_e = m_{ij} = m_{ji} =
q$.}
\label{Fig2}
\end{figure}

Then a straightforward computation shows that even for the two
non-isomorphic graphs with four nodes with edge-data as above, neither
$\det(\dt)$ nor $\cof(\dt)$ agree (unless $q=1$). \qed
\end{example}

Our next main result for trees provides a closed-form expression for
$\dt^{-1}$ in our general setting, thereby subsuming the special cases
worked out in~\cite{BKN}--\cite{BS2}, \cite{GL,ZD1,ZD2}.

\begin{utheorem}\label{Tinverse}
Suppose $T = (V = [n], E)$ is equipped with edge-data $\mathcal{T} = \{
(a_{ij} = a_{ji}, m_{ij}, m_{ji}) : \{ i, j \} \in E \}$, such that
$a_{ij}, m_{ij} m_{ji} - 1$, and $\det(\dt)$ are invertible in $R$.
Define the vectors $\tauin, \tauout \in R^V$ to have $i$th coordinates:
\begin{equation}\label{Etauinout}
\tauin(i) := 1 - \sum_{j : j \sim i} \frac{m_{ji} (m_{ij} - 1)}{m_{ij}
m_{ji} - 1}, \qquad
\tauout(i) := 1 - \sum_{j : j \sim i} \frac{m_{ij} (m_{ji} -
1)}{m_{ij} m_{ji} - 1},
\end{equation}
and also define the {\em Laplacian matrix} $L_\mathcal{T} \in R^{|V|
\times |V|}$ via:
\begin{equation}
(L_\mathcal{T})_{ij} = \begin{cases}
\displaystyle \frac{-m_{ij}}{a_{ij} (m_{ij} m_{ji} - 1)}, \qquad &
\text{if } i \sim j;\vspace*{3mm}\\
\displaystyle \sum_{k \sim i} \frac{m_{ki}}{a_{ik} (m_{ik} m_{ki} - 1)},
& \text{if } i = j;\vspace*{2mm}\\
0, & \text{otherwise}.
\end{cases}
\end{equation}
Then there exists a matrix $C_\mathcal{T} \in R^{|V| \times |V|}$
(see~\eqref{EBmatrix} below) such that
\begin{equation}
\dt^{-1} = \left( \sum_{e \in E} \frac{a_e (m_e - 1)(m'_e - 1)}{m_e m'_e
- 1} \right)^{-1} \tauout \tauin^T - L_\mathcal{T} + C_\mathcal{T}
\diag(\tauin).
\end{equation}
\end{utheorem}

This result is proved in Section~\ref{Sinverse}, and in a sense is the
strongest result in the paper, since as we explain, it -- and its proof
-- implies most of Theorem~\ref{Tmaster}, and hence all preceding
results.\medskip

Having dealt with trees, our final (until the last section) main result
concerns distance matrices of arbitrary finite directed, strongly
connected, weighted graphs. We mention some definitions for completeness,
referring the reader to e.g.~\cite{Harary} for the basics of graph
theory.

\begin{definition}
Let $G$ be a directed or undirected graph with node-set $V$.
\begin{enumerate}
\item A node $v \in V$ is said to be a 
\textit{cut-vertex} if its removal increases the number of (undirected)
graph components of $G$. Denote the set of cut-vertices by $V^{cut}$.

\item The maximal induced subgraphs of $G$ with no cut-vertices are
called the \textit{strong blocks}.

\item If $G$ is undirected (directed), we say $G$ is \textit{(strongly)
connected} if for all $v \neq w \in V$ there exists a (directed) path
from $v$ to $w$.
\end{enumerate}
\end{definition}

Now suppose $G$ is a finite directed, strongly connected, weighted graph.
Distance matrices $D_G$ with trivial multiplicative edgeweights --
i.e.~$m_e = m'_e = q\ \forall e \in E(G)$ and $q \to 1$ as in the setting
of Theorem~\ref{Tblp} -- were studied by
Graham--Hoffman--Hosoya~\cite{GHH}, and they obtained beautiful formulas
for $\det(D_G), \cof(D_G)$ in terms of the strong blocks of $G$:
\begin{align}\label{Eghh}
\begin{aligned}
\cof(D_G) = &\ \prod_j \cof(D_{G_j}),\\
\det(D_G) = &\ \sum_j \det(D_{G_j}) \prod_{i \neq j} \cof(D_{G_i}).
\end{aligned}
\end{align}
Here $G_j$ (and $G_i$) run over the strong blocks of $G$. In particular
when $\cof(D_G) \neq 0$, one has:
\[
\frac{\det(D_G)}{\cof(D_G)} = \sum_j \frac{\det(D_{G_j})}{\cof(D_{G_j})}.
\]

We next present similar formulas for $\cof(D_G)$ and $\det(D_G)$ in the
parallel multiplicative setting. More generally, we now work with
distance matrices whose $(i,j)$ entries are themselves matrices:

\begin{definition}
Fix a unital commutative ring $R$ and a directed, strongly connected
graph $G$ with vertex set $V = [n]$. For us, a \textit{product distance}
on $G$ is a choice of integers $k_1, \dots, k_n \geq 1$ and matrices
$\eta(i,j) \in R^{k_i \times k_j}$ such that
(a)~$\eta(v,v) := \Id_{k_v}$ for all cut-vertices $v$, and
(b)~for $i,j \in [n]$, if every directed path from $i \to j$ passes
through the cut-vertex $v$, then $\eta(i,j) = \eta(i,v) \eta(v,j)$.
Here the \textit{product distance matrix} is the $K \times K$ block
matrix with $(i,j)$ block $\eta(i,j)$, where $K := \sum_{v \in V} k_v$.
\end{definition}

Product distance matrices have been previously studied -- mostly for
trees~\cite{YY2,ZD1}, but also in~\cite{BS2} for general graphs, with
$k_i = 1\ \forall i$. 
%To our knowledge, the invariant $\cof(\cdot)$ was not computed in these
%settings, hence also not a `multiplicative' $\det$-$\cof$ formula in the
%spirit of~\cite{GHH}.
The above definition simultaneously extends the settings in all of these
works. Now in this overarching setting -- and for arbitrary graphs -- we
will show:

\begin{utheorem}\label{Tghhmult}
Suppose $G = (V,E)$ is a finite directed, strongly connected, weighted
graph, with additive edgeweights $a_e = a'_e = 1\ \forall e \in E$.
Suppose $G$ has strong blocks $G_j$, and $D^*_G$ denotes any product
distance matrix for $G$, with principal submatrices $D^*_{G_j}$
corresponding to $G_j$. Then,
\begin{align}\label{Eghhdetcof}
\begin{aligned}
\det(D^*_G) = &\ \prod_j \det(D^*_{G_j}),\\
\cof(D^*_G) = &\ \sum_j \cof(D^*_{G_j}) \prod_{i \neq j} \det(D^*_{G_i})
- \det(D^*_G) \sum_{v \in V} k_v (\# \{j : v \in G_j\} - 1),
\end{aligned}
\end{align}
where the final sum may be taken over only the subset of cut-vertices. In
particular, and parallel to the setting of~\cite{GHH}, if $D^*_G$ is
invertible, and the integers $k_v = k\ \forall v \in V^{cut}$ are all
equal, then
\begin{equation}\label{Eghh1}
\frac{\cof(D^*_G)}{\det(D^*_G)} - k= \sum_j \left(
\frac{\cof(D^*_{G_j})}{\det(D^*_{G_j})} - k \right).
\end{equation}

Also: if $D^*_{G_j}$ is invertible for all $j$, then
\begin{equation}\label{Eghhinv}
(D^*_G)^{-1} = \sum_j \left[ (D^*_{G_j})^{-1} \right]_j - \sum_{v \in V}
(\# \{j : v \in G_j\} - 1) \cdot [\Id_{k_v}]_v,
\end{equation}
where $[A]_j$ denotes the $K \times K$ matrix with the matrix $A$
occurring in the rows and columns corresponding to the nodes of $G_j$,
and zeros in the other entries (and similarly for $[\Id_{k_v}]_v$).
\end{utheorem}

In the special case $k_v = 1\ \forall v$, the formulas here for $\det,
\cof$ resemble those in~\eqref{Eghh}. In fact our results strengthen the
classical Graham--Hoffman--Hosoya identities~\eqref{Eghh} in multiple
ways: first, we show in the next section how our multiplicative formulas
imply the additive ones in~\eqref{Eghh}.
Second, in the final Section~\ref{Sthird} we propose (and prove) similar
formulas to Theorem~\ref{Tghhmult} in our current, general setting --
with scalar entries; and then show how these too specialize to the
identities~\eqref{Eghh}. This uses a third, novel invariant
$\kappa(\dt)$.

\begin{remark}
For example, Theorem~\ref{Tghhmult} implies explicit formulas for
$\det(D^*_T)$, $\cof(D^*_T)$, $(D^*_T)^{-1}$ for $G = T$ an arbitrary
tree, as in~\cite{BR,YY2,ZD1}.
This is because the strong blocks of $T$ are precisely its edges, and it
is easy to compute the above quantities for $2 \times 2$ block matrices
-- in fact of the form $\begin{pmatrix} \Id_{k_1} & M_{12} \\ M_{21} &
\Id_{k_2} \end{pmatrix}$. Note, the explicit formulas for $\det(D^*_T),
(D^*_T)^{-1}$ for trees in~\cite{BR,YY2,ZD1} -- for distance matrices
$D^*_T$ with matrix weights as above -- are stated in a different form.
\end{remark}

\begin{remark}
In this paper, we mainly focus on advancing the state-of-the-art
(e.g.~non-principal minors) for matrices $\dt$ over arbitrary unital
commutative rings $R$. Indeed, the commutative case is the far
better-studied and mathematically active framework for distance matrices
and their variants. That said, there are also results in the literature
over non-commutative rings $R'$. See e.g.~\cite{BS,ZD1,ZD2}, several of
which work in fact with $R' = R^{k \times k}$ -- precisely the setting in
our Theorem~\ref{Tghhmult}.
\end{remark}

For completeness, we conclude by mentioning a generalization of the
formula~\eqref{Eghh1} for invertible $D^*_G$. Namely, fix a cut-vertex
$v$. Then every block $G_j$ has a unique cut-vertex $v(j)$ which is
closest to $v$. Now~\eqref{Eghh1} extends to the case of possibly unequal
$k_v$ as follows:
\begin{equation}\label{Eghh2}
\frac{\cof(D^*_G)}{\det(D^*_G)} - k_v = \sum_j \left(
\frac{\cof(D^*_{G_j})}{\det(D^*_{G_j})} - k_{v(j)} \right), \qquad
\forall v \in V^{cut}.
\end{equation}

\subsection*{Organization of the paper}

In Section~\ref{Sghhmult} we prove Theorem~\ref{Tghhmult}, followed by
several applications:
\begin{itemize}
\item The classical Graham--Hoffman--Hosoya formulas~\eqref{Eghh} -- more
generally, a novel $q$-variant.

\item Attaching finitely many pendant trees to $G$, and showing the
independence of $\det, \cof$ from the locations where these are attached;

\item Computing these invariants for the classical and $q$-weighted
cycle-clique graphs, thereby recovering known results for unicyclic,
bicyclic, and cactus graphs, as well as hypertrees.
\end{itemize}

In Section~\ref{Sinverse} we prove Theorem~\ref{Tinverse}, and use it to
show Theorem~\ref{Tmaster}. We add that one can provide other proofs of
Theorem~\ref{Tmaster} or its special cases that are not in the
literature, but we omit these for brevity. Also remark that the
multiplicative special case of $a_e = a'_e = 1$ (proved in
Section~\ref{Sghhmult}) is required in order to prove
Theorem~\ref{Tmaster} via a novel technique that we introduce in this
field: applying Zariski density to distance matrices; to our knowledge,
this is new to the area. Namely:
\begin{itemize}
\item we note that $\det(\dt)_{I|J'}, \cof(\dt)_{I|J'}$, and also our
stated formulas for them in Theorem~\ref{Tmaster} are \textit{polynomial
functions} of the matrix entries in $(\dt)_{I|J'}$;
\item hence we first work with variable edgeweights and
use Zariski density -- since $\det(\dt), 1 - m_e m'_e \not\equiv 0$ from
the multiplicative case, proved independently. This shows the result over
$\Z[\{ a_e, m_e, m'_e : e \in E \}]$; we then specialize to an arbitrary
unital commutative ring $R$.
\end{itemize}
This technique is immensely useful in our proofs -- see also
Remark~\ref{Rzariski} for some of the advantages.

In the final Section~\ref{Sthird}, we introduce a \textit{third}, novel
invariant $\kappa(\dt)$ for the above ``general'' distance matrices $\dt$
for trees~\eqref{Eedgeweight}.
We show that $\kappa$ is also independent of the tree structure, and is
multiplicative across edges; and for general graphs we prove
Graham--Hoffman--Hosoya type identities for $\det(D_G), \cof(D_G),
\kappa(D_G)$ in terms of the strong blocks of $G$ -- see
Theorem~\ref{Tghh3}. This provides an alternate, short proof of the
formulas for $\det(\dt), \cof(\dt)$ in the general setting of
Theorem~\ref{Tmaster} -- from which all known formulas for
$\det(\dt), \cof(\dt)$ for trees can be deduced.
%}}}

%{{{1 Section 2 - The multiplicative GHH Theorem C for digraphs, and its applications
\section{The multiplicative GHH Theorem~\ref{Tghhmult} for
digraphs, and its applications}\label{Sghhmult}

In this section we prove Theorem~\ref{Tghhmult} for product distance
matrices over arbitrary weighted strongly connected graphs, and provide
several applications to graphs that are not trees. The final two sections
will focus on weighted bi-directed trees.

\begin{proof}[Proof of Theorem~\ref{Tghhmult}]
Inducting on the number of cut-vertices, it suffices to show the result
for $G$ having a cut-vertex $v \in V = [n]$ and consisting of strongly
connected subgraphs $G_1$ and $G_2$ separated by $v$, with node sets $\{
1, \dots, v \}$ and $\{ v, \dots, n \}$ respectively. Thus $G$ has
distance matrix
\begin{equation}\label{Eghhmatrix}
D^*_G = \begin{pmatrix}
D_1 & A & AB\\
A' & \Id_{k_v} & B\\
B'A' & B' & D_2
\end{pmatrix},
\end{equation}
where the leading (resp.~trailing) principal $2 \times 2$ block submatrix
equals $D^*_{G_1}$ (resp.~$D^*_{G_2}$).

We first quickly sketch the argument for computing $\det(D^*_G)$, since
it is similar to the one for various special cases shown
in~\cite{BS2,YY2,ZD1}. By elementary linear algebra it is clear that
\[
\det D^*_G = \det \begin{pmatrix}
D_1 & A & 0\\
A' & \Id_{k_v} & 0\\
B'A' & B' & D_2 - B'B
\end{pmatrix} = \det D^*_{G_1} \cdot \det (D_2 - B' \Id_{k_v}^{-1} B) =
\det D^*_{G_1} \cdot \det D^*_{G_2},
\]
where the final equality uses Schur complements. This proves the
assertion.

Next, if $D^*_G$ is invertible (i.e., $\det(D^*_{G_j}) \in R^\times$ for
all $j$) then an explicit computation shows the claimed formula for its
inverse. Once again, it suffices by induction to work with $G$ having one
cut-vertex $v \in [n]$ as above, in which case if $D^*_G$ is of the
form~\eqref{Eghhmatrix}, then
\begin{align*}
&\ D^*_G \cdot \begin{pmatrix} (D^*_{G_1})^{-1} & 0 \\ 0 & 0
\end{pmatrix} + D^*_G \cdot \begin{pmatrix} 0 & 0 \\ 0 & (D^*_{G_2})^{-1}
\end{pmatrix} - D^*_G \cdot \begin{pmatrix} 0 & 0 & 0 \\ 0 & \Id_{k_v} &
0 \\ 0 & 0 & 0 \end{pmatrix}\\
= &\ \begin{pmatrix} \Id_{|V(G_1)|} & 0 \\ B' (0 \quad \Id_{k_v}) & 0
\end{pmatrix} + \begin{pmatrix} 0 & A (\Id_{k_v} \quad 0) \\ 0 &
\Id_{|V(G_2)|} \end{pmatrix} - \begin{pmatrix} 0 & A & 0 \\ 0 & \Id_{k_v}
& 0 \\ 0 & B' & 0 \end{pmatrix} = \Id_{|V(G)|}.
\end{align*}

Finally, if $D^*_G$ is invertible (equivalently, all $D^*_{G_j}$ are
thus), then pre- and post- multiplying~\eqref{Eghhinv} by ${\bf e}^T,
{\bf e}$ respectively yields via Lemma~\ref{Ldetcof}:
\[
\frac{\cof(D^*_G)}{\det(D^*_G)} = \sum_j
\frac{\cof(D^*_{G_j})}{\det(D^*_{G_j})} - \sum_{v \in V^{cut}} (\# \{ j :
v \in G_j \} - 1) \cdot k_v.
\]
Now the claimed identity for $\cof(D^*_G)$ follows from the one for
$\det(D^*_G)$ -- note this holds whenever $\det(D^*_G)$ is invertible. To
prove this holds uniformly, we use a Zariski density argument (or Weyl's
``principle of irrelevance of algebraic inequalities''~\cite{W}). More
precisely, we first work over the field $R_0 := \mathbb{Q}(\{ a_{i,i'}
\})$ of rational functions, where $i,i'$ together index rows (or columns)
of the block matrix $D^*_G$ which belong to the block-entries of unequal
nodes in a common block $D^*_{G_j}$.

Now in the field $R_0$, $\det(D^*_G)$ is a nonzero polynomial (hence
invertible), since specializing to $\eta(i,j) = 0\ \forall i \neq j$
yields $D^*_G = \Id$. In particular, it follows that
\[
p := \cof(D^*_G) - \sum_j \cof(D^*_{G_j}) \prod_{i \neq j}
\det(D^*_{G_i}) + \det(D^*_G) \sum_{v \in V} k_v (\# \{ j : v \in G_j \}
- 1)
\]
is a polynomial in $\Z[ \{ a_{i,i'} \}] \subset R_0$, which vanishes on
the set $U := D(\det(D^*_G))$ of non-roots of $\det(D^*_G)$. Now invoke
Lemma~\ref{Lzariski}, stated and proved below:
$U$ is Zariski dense in the affine space $\mathbb{A}_\mathbb{Q}^N$, where
$N := {\rm tr} \, {\rm deg}_{\mathbb{Q}}(R_0)$ denotes the number of
variables $a_{i,i'}$ in $R_0$. Thus $p$ vanishes on all of
$\mathbb{A}_\mathbb{Q}^N$, and the formula for $\cof(D^*_G)$ holds
uniformly for all values of $(a_{i,i'}) \in \mathbb{Q}^N$. Again using
Lemma~\ref{Lzariski} shows $p = 0$ in $\Z[\{ a_{i,i'} \}]$, so one can
specialize the variables $a_{i,i'}$ to any unital commutative ring $R$,
to prove the formula for $\cof(D^*_G)$ over $R$. Now~\eqref{Eghh1}
follows from \eqref{Eghhinv}.
\end{proof}

Notice that modulo a Zariski density lemma, the final two paragraphs of
our proof of Theorem~\ref{Tghhmult} spelled out the ``Zariski density''
part of the proof in great detail; this was because we will use similar
arguments in what follows. The remaining piece is a basic and well-known
lemma on polynomials:

\begin{lemma}\label{Lzariski}
Suppose $\mathbb{F}$ is an infinite field and $k > 0$ an integer.
\begin{enumerate}
\item If a polynomial $p \in \mathbb{F}[T_1, \dots, T_k]$ vanishes on
$\mathbb{F}^k$, then $p = 0$.

\item For a polynomial $0 \neq f \in \mathbb{F}[T_1, \dots, T_k]$,
let $D(f) := \{ (a_1, \dots, a_k) \in \mathbb{A}_\mathbb{F}^k \cong
\mathbb{F}^k : f(a_1, \dots, a_k)$ $\neq 0 \}$. Then $D(f)$ -- in fact,
every nonempty Zariski open set -- is Zariski dense in $\mathbb{F}^k$.
(Recall, $U \subseteq \mathbb{F}^k$ is Zariski dense if 
$p|_{\mathbb{F}^k} \equiv 0$ whenever $p \in \mathbb{F}[T_1, \dots, T_k]$
vanishes on $U$.)

\item Consider the following inductively indexed family of sets, each of
which is infinite in size:
(a)~$S \subseteq \mathbb{F}$ is an infinite subset;
(b)~for each $j \in [k-1]$, given $s_1 \in S, s_2 \in S_{s_1}, \dots, s_j
\in S_{s_1, \dots, s_{j-1}}$,
let $S_{s_1, \dots, s_j} \subseteq \mathbb{F}$ be an infinite subset.
Then the set of tuples
\[
\mathcal{S}(k) := \{ {\bf s} = (s_1, \dots, s_k) \in \mathbb{F}^k : s_1
\in S, \ s_{j+1} \in S_{s_1, \dots, s_j}\ \forall j \in [k-1] \}
\]
is Zariski dense in $\mathbb{A}_{\mathbb{F}}^k \cong \mathbb{F}^k$.
\end{enumerate}
\end{lemma}

For instance in part~(3), one can choose $S_{s_1, \dots, s_j} = S\
\forall j, \ s_1, \dots, s_{j-1}$, in which case $\mathcal{S}(k) := S^k$.
Also note that a strengthening of part~(3) is well-known, e.g.\ see
\cite[Lemma 2.1]{Alon} in Alon's famous work. Moreover, part~(2) is a
(weaker) version of Weyl's principle of irrelevance of algebraic
inequalities~\cite{W}.

\begin{remark}\label{Rghhspecial}
For completeness, we discuss various special cases of
Theorem~\ref{Tghhmult} in the literature. The formula for
$\det(D^*_G)$ 
was recently obtained for matrices with all $k_i \equiv 1$ in~\cite{BS2}
in the spirit of previous results in~\cite{YY2,ZD1}. Also, with
Lemma~\ref{Ldetcof} in hand, it is not hard to observe that the
multiplicative ``$q$-results'' -- analogous to the classical
Graham--Hoffman--Hosoya identities~\eqref{Eghh} -- which were obtained
recently in~\cite{LSZ,S} can be derived from Theorem~\ref{Tghhmult}, by
specializing to $m_e = m_{ij} = m_{ji} = q$ and $a_e = a_{ij} = a_{ji} =
1/(q-1)$. In these works~\cite{LSZ,S}, the authors consider $\det(D^*_G -
J) = (\det-\cof)(D^*_G)$ and $\xi(D^*_G - J) = \det(D^*_G)$ (up to a
power of $q-1$).

Also notice, Theorem~\ref{Tghhmult} immediately implies the independence
of $\cof(D^*_T)$ from the tree structure (for $G$ a tree), since now the
strong blocks are precisely the edges of $G$.
Moreover, the entries $\eta(i,j)$ of the distance matrix $D^*_G$ can be
matrices and we allow $\eta(i,j) \neq \eta(j,i)$, extending the
scalar-entries case in the works~\cite{BS2,LSZ,S} cited above.
\end{remark}

\subsection{Application 1: the $q$- and classical Graham--Hoffman--Hosoya
identities}

Before focussing on trees below, we discuss some applications of
Theorem~\ref{Tghhmult}. 
Our first application shows that Theorem~\ref{Tghhmult} implies analogous 
identities for the \textit{$q$-distance matrix} that (we believe) were
not yet written down. Consider a weighted bi-directed graph $G$ with
$d(u,v) = \alpha_{uv} \in \Z$, say; now set
\begin{equation}\label{Eqmatrix}
(D_q(G))_{u,v} = [\alpha_{uv}] := \frac{q^{\alpha_{uv}} - 1}{q-1} = 1 + q
+ \cdots + q^{\alpha_{uv} - 1},
\end{equation}
where $q$ is a formal parameter. Thus we work over the field
$\mathbb{Q}(q)$. Notice that
\begin{equation}\label{EDq}
D_q(G) = (q-1)^{-1} (D^*_q(G) - J),
\end{equation}
where $D^*_q(G)$ is the ``$q$-multiplicative'' $V \times V$ matrix with
$(u,v)$ entry $q^{d(u,v)}$; now specializing $q \to 1$ (for integer
entries $d(u,v)$) yields precisely the classical distance matrix
$(d(u,v))_{u,v}$. Thus, $D_q(G)$ is an ``intermediate'' matrix used to
pass from $D^*_q(G)$ (or more generally, the multiplicative variant
$D^*_G$) to $D_G = D_1(G)$, and we have Graham--Hoffman--Hosoya type
formulas~\eqref{Eghh} and~\eqref{Eghhdetcof} for the two ``endpoints'' of
this procedure:
\begin{align*}
\det(D^*_G) = &\ \prod_j \det(D^*_{G_j}),\\
\cof(D^*_G) = &\ \prod_j \det(D^*_{G_j}) + \sum_j (\cof(D^*_{G_j}) -
\det(D^*_{G_j})) \prod_{i \neq j} \det(D^*_{G_i}),\\
\det(D_G) = &\ \sum_j \det(D_{G_j}) \prod_{i \neq j} \cof(D_{G_i}),\\
\cof(D_G) = &\ \prod_j \cof(D_{G_j}).
\end{align*}

\noindent The following result presents the corresponding formulas for
the intermediate matrices $D_q(G)$:

\begin{prop}[$q$-GHH identities]\label{Pghh2ghh}
Say $G$ is a directed strongly connected graph with node set $V$, and
$D \in \Z^{V \times V}$ is any matrix such that $d(v,v) = 0\ \forall v
\in V$, and if every directed path from $u \to v$ passes through the cut
vertex $v_0$, then $d(u,v) = d(u,v_0) + d(v_0,v)$.
Now let the $V \times V$ matrix $D_q(G)$ have $(u,v)$ entry $[d(u,v)] =
(q-1)^{-1} (q^{d(u,v)} - 1)$.
If $G$ has strong blocks $G_j$ with corresponding principal
submatrices $D_q(G_j)$, then defining $d^*_j := (q-1) \det(D_q(G_j)) +
\cof(D_q(G_j))$, we have:
\begin{align}\label{Eghhq}
\begin{aligned}
\det(D_q(G)) = &\ \sum_j \det(D_q(G_j)) \prod_{i \neq j} d^*_i,\\
\cof(D_q(G)) = &\ \prod_j d^*_j - (q-1) \sum_j \det(D_q(G_j)) \prod_{i
\neq j} d^*_i.
\end{aligned}
\end{align}
\end{prop}

The proof is omitted, as it is merely a computation using
Theorem~\ref{Tghhmult} (and Lemma~\ref{Ldetcof}), via~\eqref{EDq}.

\begin{remark}
The restriction that $d(u,v) \in \Z$ in $D_q(G)$ was merely in order to
work in a familiar setting. However, Proposition~\ref{Pghh2ghh} allows
$d(u,v)$ to take values in an arbitrary abelian group $\Gamma$; it would
then hold in the group algebra $\Z[q^\Gamma] \cong \Z[\Gamma]$. Once
again, these formulas can be shown using the polynomiality of
$\det(D_q(G))$ and $\cof(D_q(G))$ in the matrix entries.
\end{remark}

\begin{remark}
In addition to implying the additive Graham--Hoffman--Hosoya
identities~\eqref{Eghh}, our $q$-variant also implies that $\det(D_q(G)),
\cof(D_q(G))$, and $\cof(D_q(G)) / \det(D_q(G)) = {\bf e}^T D_q(G)^{-1}
{\bf e}$ depend only on the corresponding quantities for the strong
blocks $G_j$ of $G$, and not on the block/tree structure of $G$. This
immediately proves all such observations made in the literature, see
e.g.~\cite{S-hypertrees}, and has perhaps an advantage over the variants
in the literature that use $q$-cofsum$(D_q(G))$ or $\xi(D_q(G))$, since
deducing from the latter the invariance of ${\bf e} D_q(G)^{-1} {\bf e}$
is not as immediate.
\end{remark}

In turn, Proposition~\ref{Pghh2ghh} leads to (a novel proof of) the
original additive GHH-identities:

\begin{prop}\label{Tghh2ghh}
Notation as in Proposition~\ref{Pghh2ghh}, except that now $D$ has
entries $d(u,v)$ in a general unital commutative ring $R$. If $G_j$
denote the strong blocks of $G$, with corresponding principal submatrices
$D_{G_j}$, then the identities~\eqref{Eghh} follow from
Theorem~\ref{Tghhmult} (via Proposition~\ref{Pghh2ghh}):
\[
\cof(D_G) = \prod_j \cof(D_{G_j}), \qquad
\det(D_G) = \sum_j \det(D_{G_j}) \prod_{k \neq j} \cof(D_{G_k}).
\]
\end{prop}

\begin{proof}
The result for integer values of $d(u,v)$ is immediate by setting $q=1$
in Proposition~\ref{Pghh2ghh}, since then $d_j^* = \cof(D_{G_j})$.
Now to go from such matrices with integer entries to arbitrary unital
commutative rings uses a Zariski density argument, given that $\det(D_G),
\cof(D_G)$ are polynomials in the matrix entries in $E := \bigsqcup_k
\overset{\longrightarrow}{E}(G_k)$, where
$\overset{\longrightarrow}{E}(G_k)$ runs over the directed edges of the
strong block $G_k$ of $G$. We work over the ring $R_0 := \mathbb{Q}(E)$,
and denote by $p_d, p_c \in \Z[E]$ the polynomials
\[
p_d := \det(D_G) - \sum_j \det(D_{G_j}) \prod_{i \neq j}
\cof(D_{G_i}), \qquad
p_c := \cof(D_G) - \prod_j \cof(D_{G_j}).
\]
Since $p_d, p_c$ vanish on $\Z_{>0}^E$, which is Zariski dense in
$\mathbb{Q}^E$ by Lemma~\ref{Lzariski}(3), it follows that $p_d = p_c =
0$ in $\Z[E]$. The proof concludes by specializing to an arbitrary unital
commutative ring.
\end{proof}

\begin{remark}
Suppose $d(u,v) \in \Z \ \forall u,v \in V$, and define the matrix
$D^*_q(G)$ to have $(u,v)$ entry $q^{d(u,v)}$, where $q$ is a parameter
as above. Then it is possible to jump directly from $D^*_q(G)$ to
$D_G$:
\begin{align}\label{Eghh2ghh}
\begin{aligned}
\lim_{q \to 1} (q-1)^{-|V|} (\det - \cof)(D^*_q(G)) = &\ \det (D_G),\\
\lim_{q \to 1} (q-1)^{1-|V|} \det(D^*_q(G)) = &\ \cof (D_G),
\end{aligned}
\end{align}
where $q \to 1$ stands for setting $q=1$ after dividing by the relevant
power of $q-1$. Notice that while the first of these formulas can be
anticipated from previous papers (see Remark~\ref{Rghhspecial}), the
latter formula is new, or at least not immediately clear. Also, the
specializations in~\eqref{Eghh2ghh} mean that Theorem~\ref{Tghhmult} is
more general than its classical, additive version in~\cite{GHH}.

To show~\eqref{Eghh2ghh}, we again appeal to $D_q(G)$. The first of the
identities~\eqref{Eghh2ghh} follows immediately from~Lemma~\ref{Ldetcof},
since $D^*_q(G) = (q-1) D_q(G) + J$. This last equality also
implies
\[
(q-1)^{|V|} \det(D_q(G)) + (q-1)^{|V|-1} \cof(D_q(G)) = \det (D^*_q(G)),
\]
via Lemma~\ref{Ldetcof}. Subtracting from this the first of the
formulas~\eqref{Eghh2ghh} (times $(q-1)^{|V|}$) yields:
\[
(q-1)^{|V|-1} \cof(D_q(G)) = \cof(D^*_q(G)),
\]
so that $\cof(D_G) = \lim_{q \to 1} (q-1)^{1-|V|} \cof(D^*_q(G))$. Now
the first identity in~\eqref{Eghh2ghh} implies the claim:
\[
\lim_{q \to 1} (q-1)^{1-|V|} \cof(D^*_q(G)) = \lim_{q \to 1}
(q-1)^{1-|V|} \det(D^*_q(G)).
\]
\end{remark}

\subsection{Application 2: inverse identities}

In several papers in the literature (see the remarks prior to
Theorem~\ref{Tinverse} for trees; but also
e.g.~\cite{cactoid,cycle-clique}), the inverse of distance matrices of
graphs is computed. These are additive matrices of specific graphs $G$,
and we now discuss a recipe to obtain such formulas for general strongly
connected $G$.

Recall that the previous subsection discussed formulas for $\det(\cdot)$
and $\cof(\cdot)$ for strongly connected graphs in terms of their strong
blocks -- for both the $q$- and classical distance matrices. These
formulas were obtained as consequences of Theorem~\ref{Tghhmult} -- which
contains identities for $\det(D^*_G)$ and $\cof(D^*_G)$, but also for
$(D^*_G)^{-1}$. In that spirit, we now record the analogous identity for
$D_q(G)^{-1}$ in terms of the matrices $D_q(G_j)^{-1}$.
(This is done for completeness, and in slightly greater generality; we
leave to the interested reader the explicit such identity, as well as its
specialization to $q=1$.)

\begin{prop}\label{Pmult-inv}
Notation as in Theorem~\ref{Tghhmult}; assume $k_v = 1\ \forall v \in V$.
Let
\[
{\bf u}_j := [(D^*_{G_j})^{-1}]_j [ {\bf e}(|V(G_j)|)]_j \in R^V,
\]
where $[{\bf v}]_j$ denotes the $|V| \times 1$ vector in which the
entries of ${\bf v} \in R^{|V(G_j)|}$ occur in the rows
corresponding to the nodes of $G_j$, with all other entries zero;
and similarly, $[A]_j$ denotes the $|V| \times |V|$ matrix with the
matrix $A \in R^{|V(G_j)| \times |V(G_j)|}$ occurring in the rows and
columns corresponding to the nodes of $G_j$, with all other entries
zero. Also let ${\bf e}_{cut} \in R^V$ denote the $\{ 0, 1\}$-vector
with ones in precisely the coordinates corresponding to $V^{cut}$. Then
for $x$ an indeterminate over $R$,
\begin{align}
\begin{aligned}
(D^*_G + xJ)^{-1} = &\ \sum_j \left( \left[ (D^*_{G_j} + xJ)^{-1}
\right]_j + \frac{x \det(D^*_{G_j})}{\det(D^*_{G_j} + x J)} {\bf u}_j
{\bf u}_j^T \right) - \sum_{v \in V^{cut}} E_{v,v}\\
&\ - \frac{x \det(D^*_G)}{\det(D^*_G + x J)}  \left( -{\bf e}_{cut} +
\sum_j {\bf u}_j \right) \left( -{\bf e}_{cut} + \sum_j {\bf u}_j
\right)^T.
\end{aligned}
\end{align}
\end{prop}

Here $E_{v,v}$ is the elementary matrix with $(i,j)$ entry $\delta_{i,v}
\delta_{v,j}$. The proof involves carefully applying the above results as
well as the Sherman--Morrison formula for the inverse of a rank-one
update.

The point is that this result provides a closed-form expression for the
inverse of $\tild_G := D^*_G + xJ$ in terms of the inverse and invariants
for $\tild_{G_j} := D^*_{G_j} + xJ$. Indeed, one now writes $D^*_{G_j} =
\tild_{G_j} - xJ$, and using Theorem~\ref{Tghhmult}, one also writes
$\det(D^*_G), \det(D^*_G + xJ)$ in terms of $\det(\tild_{G_j}),
\cof(\tild_{G_j})$. Moreover, by the final assertion in Lemma~\ref{Ldetcof},
${\bf u}_j$ can also be written purely in terms of $\tild_{G_j}$:
\[
{\bf u}_j = \frac{\det(\tild_{G_j})}{\det(\tild_{G_j} - xJ)} [
(\tild_{G_j} - xJ)^{-1} ]_j [{\bf e}(|V(G_j)|)]_j.
\]
Now if $D_q(G)$ denotes the usual $q$-distance matrix for $G$, and $D^*_G
= D^*_q(G)$, then $D_q(G) = (q-1)^{-1} \tild_G|_{x=-1}$. This and the
above result allow one to compute $D_q(G)^{-1}$. Specializing to
$q=1$ yields a formula for $D_G^{-1}$ that subsumes special cases
in~\cite{BKN}--\cite{BS2}, \cite{GL,cactoid,cycle-clique,ZD1,ZD2}.

\begin{remark}
Proposition~\ref{Pmult-inv} can be further generalized to the setting of
Theorem~\ref{Tghhmult}, in which $D^*_G$ is no longer $|V| \times |V|$.
We leave to the interested reader the similar formulation and proof.
\end{remark}

\subsection{Application 3: adding pendant trees}

The next application of our above results generalizes several results in
the literature for graphs outside of trees. We provide detailed
references below; in all of them, the classical or $q$-distance matrix is
what has been studied.

As a first step, the (multiplicative and $q$-) Graham--Hoffman--Hosoya
identities in Theorem~\ref{Tghhmult} and Proposition~\ref{Pghh2ghh}
immediately give that the changes in the determinant and cofactor that
occur due to ``attaching trees to graphs'', depend only on the set of
additional nodes/edges:

\begin{cor}
If $G$ is as in Theorem~\ref{Tghhmult}, and $G'$ is obtained by attaching
finitely many pendant trees $T_j$ to the nodes of $G$, with a total of
$m$ additional bi-directed edges, then
\[
\det(D^*_{G'}), \quad \cof(D^*_{G'}), \quad \det(D_q(G')), \quad
\cof(D_q(G'))
\]
are independent of the structure or location of the $T_j$, and depend
only on $G$ and the $m$ edge-data.
\end{cor}

Consider the special case $k_v = 1$ for all nodes $v$ of $G$. Then we are
dealing with the $q$-distance matrix, and one can use
Proposition~\ref{Pghh2ghh} and $\det(D_q(T)), \cof(D_q(T))$ for $T$ each
attached edge:

\begin{prop}\label{Paddtrees}
Let $k \geqslant 1$ be an integer, and for each $j \in [k]$ let $G_j$ be
a weighted bi-directed graph on $p_j$ nodes, with $q$-distance matrix
$D_q(G_j)$ that satisfies:
\[
D_q(G_j) {\bf e}(p_j) = d_j {\bf e}(p_j), \qquad \forall j \in [k],
\]
where $d_j \in R$ are scalars.
Now let $G$ be any strongly connected graph with strong blocks $G_1,
\dots, G_k$, and let $G'$ be obtained from $G$ by further attaching
finitely many pendant trees to the nodes of $G$, with a total of $m$ new
vertices and hence $m$ new bi-directed edges $E$ with weights $\{
(\alpha_e, \alpha'_e) : e \in E \}$.

Then $\det(D_q(G')), \cof(D_q(G'))$ depend not on the structure and
location of the attached trees, but only on their edge-data, and as
follows:
\begin{align}\label{Eaddtrees}
\det(D_q(G') + x J) = &\ \prod_{j=1}^k \det(D_q(G_j)) \prod_{j=1}^k (q-1 +
(p_j/d_j)) \prod_{e \in E} (-[\alpha_e + \alpha'_e]) \times\\
& \qquad \qquad \times \left[ x + (1 - (q-1)x) \left(
\sum_{j=1}^k\frac{1}{q-1 + (p_j/d_j)} + \sum_{e \in E} \frac{[\alpha_e]
[\alpha'_e]}{[\alpha_e + \alpha'_e]} \right) \right]. \notag
\end{align}
\end{prop}

We omit the proof as it is a straightforward consequence of
Proposition~\ref{Pghh2ghh} and Lemma~\ref{Ldetcof}.

\begin{remark}
Proposition~\ref{Paddtrees} simultaneously generalizes a host of results
in the literature. First, it subsumes all undirected $q$-weighted trees
-- by letting $G$ have two nodes and one edge -- hence all undirected
additively weighted trees as well (setting $q=1$).
Next, it also extends \cite[Theorem 7]{LSZ} (which addressed the special
case $\alpha_e = \alpha'_e\ \forall e$ and $x = 0$), as well as
\cite[Theorem 2.3]{BKN}, which dealt with the special case $\alpha_e =
\alpha'_e\ \forall e$ and $q=1$ (so $[\alpha] = \alpha$).
Finally, Proposition~\ref{Paddtrees} specialized to $G_1 = \cdots = G_m$
and adding \textit{no} extra trees, recovers certain results of
Sivasubramanian~\cite{S-hypertrees}.
\end{remark}

\subsection{Application 4: cycle-clique graphs and special cases}

As a final application of Theorem~\ref{Tghhmult} and its applications
above, we study \textit{cycle-clique graphs}. These are unweighted graphs
$G$ whose strong blocks are cycles and complete graphs.
Special cases include unicyclic, bicyclic, and cactus graphs -- where
each of the complete subgraphs is an edge, and the number of cycle-blocks
is one, two, or greater, respectively. For $G$ a unicyclic or bicyclic
graph, $\det(D_q(G))$ has been computed in e.g.~\cite{BKN, bicyclic,
LSZ}; and moreover, $\det(D_G)$ for an arbitrary cycle-clique graph $G$
was computed in~\cite{cycle-clique}. Another special case is the family
of \textit{regular hypertrees} -- see~\cite{S-hypertrees} -- whose
distance matrices precisely coincide with those of graphs whose strong
blocks are isomorphic cliques.

In our final application, we extend all of these formulas to arbitrary
cycle-clique graphs, in two ways. First, we allow for weighted edges; and
second, we explain how our results above help compute $\det(D_q(G) +
xJ)$. In other words, we also compute $\cof(D_q(G))$, and not only for
$q=1$.

Our setting is as follows: $G$ is a weighted strongly connected graph
with strong blocks
\[
C_{r_1}, \ \dots, \ C_{r_k}, \quad K_{p_1}, \ \dots, \ K_{p_m},
\qquad (r_j, p_i \geq 1)
\]
where $C_r, K_p$ denote the cycle and complete graphs with node sets
$[r], [p]$ respectively. Notice that under a suitable labelling of the
nodes of $C_r$, if all edges have weight $\beta$, then $D_q(C_r)$ is a
symmetric Toeplitz circulant matrix, with super/sub diagonal entries
$[\beta] = (q-1)^{-1} (q^\beta - 1)$. Denote this matrix by $D_q(C_r;
\beta)$. Then all rows and columns of $D_q(C_r; \beta)$ have the same
sum.

\begin{lemma}\label{Lcycle}
With the above notation, if $d_{r,\beta}$ is the common row/column sum of
$D_q(C_r; \beta)$, then $\cof(D_q(C_r; \beta)) = \frac{r}{d_{r,\beta}}
\det(D_q(C_r; \beta))$ for $r \geq 1$, where for $r=1$ we define
$\det(D_q(C_1; \beta)) / d_{1,\beta} := 1$.
\end{lemma}

Since $D_q(C_r; \beta)$ is a circulant matrix for each $r \geq 1$, its
determinant has a well-known closed-form expression using roots of unity
(say working over some extension of our ground ring $R$).

\begin{proof}
Apply Proposition~\ref{Paddtrees} to $G = C_r$ (with edgeweights
$[\beta]$), with $m=0$ new edges added.
\end{proof}

The other kind of blocks are the cliques (complete subgraphs) in $G$. We
first study the $q$-distance matrix of each such graph $K_p$, again under
a suitable labelling of the nodes. The following model generalizes the
$p=2$ case of bi-directed edges with weights $a_e, a'_e$:

\begin{definition}
Given an integer $p \geq 1$ and elements $a,a'$ in a unital commutative
ring, let $D(K_p; a, a')$ denote the matrix with $(i,j)$ entry $a, \ 0, \
a'$ for $i<j, \ i=j, \ i>j$ respectively. Then $D(K_p; a, a')$ is a
Toeplitz matrix, which we will denote using its entries in the first
column (bottom to top) and then the first row (left to right):
\[
D(K_p; a, a') = \Toep(a', \dots, a';\ 0;\ a, \dots, a).
\]
\end{definition}

\begin{lemma}\label{Lclique}
With the above notation,
\begin{align}
\begin{aligned}
\det D(K_p; a, a') = &\ (-1)^{p-1} a a' \, \frac{a^{p-1} -
(a')^{p-1}}{a-a'},\\
\cof D(K_p; a, a') = &\ (-1)^{p-1} \, \frac{a^p - (a')^p}{a-a'},\\
\adj D(K_p; a, a') = &\ (-1)^{p-1} \Toep( (a')^{p-1}, (a')^{p-2} a,
\dots, a' a^{p-2};\ \alpha;\ (a')^{p-2} a, \dots, a^{p-1}),
\end{aligned}
\end{align}
where $\displaystyle \alpha = -a a' \, \frac{a^{p-2} -
(a')^{p-2}}{a-a'}$ and $p \geq 2$ in the last equation; while for $p=1$
we have
\[
\adj(D(K_1; a, a')) = \cof(D(K_1; a, a')) = 1.
\]
\end{lemma}

\begin{proof}
First compute $\det D(K_p; a, a')$ by induction on $p$. Next, work over
the field $R_0 = \mathbb{Q}(a,a')$, and assume by a Zariski density
argument using Lemma~\ref{Lzariski} that $D(K_p; a, a')$ is invertible;
indeed, this is so whenever $a=a' \neq 0$, since $\det D(K_p; a, a) =
\det a (J-\Id) = (-1)^{p-1} a^p (p-1)$. This yields the expression for
$\adj(\cdot)$ via an easy verification. This shows the result over $R_0$;
now Zariski density arguments as above imply the result over any unital
commutative ring $R$.
Finally, the assertion for $\cof(\cdot)$ is shown via another computation
that uses Lemma~\ref{Ldetcof} and the formula for $\adj(\cdot)$.
\end{proof}

With these results in hand, one can show:

\begin{theorem}\label{Ppolycyclic}
Let $G$ be a connected cycle-clique graph whose strong blocks are
cycles $C_{r_1}, \dots, C_{r_k}$ and complete subgraphs $K_{p_1}, \dots,
K_{p_m}$, with all $r_j, p_i \geq 1$. Also assume there exist scalars
\[
\beta_1, \ \dots, \ \beta_k ; \quad \alpha_1, \alpha'_1, \ \dots, \
\alpha_m, \alpha'_m
\]
and a labelling of the nodes of $G$ such that the $q$-distance matrices
of the blocks are of the form
\[
D_q(C_{r_j}; \beta_j), \ j \in [k] \qquad \text{and}
\qquad D(K_{p_i}; [\alpha_i], [\alpha'_i]), \ i \in [m].
\]
\begin{enumerate}
\item Then $\det(D_q(G)), \cof(D_q(G))$ have closed-form expressions
which can be derived using Lemmas~\ref{Lcycle} and~\ref{Lclique} and
Proposition~\ref{Pghh2ghh}.

\item A sample special case: say $\beta_j = 1\ \forall j$, and $d_{r,1}$
is as in Lemma~\ref{Lcycle}. Also say $p_i = 2\ \forall i$ and denote
these $2$-cliques/edges $i \in [m]$ by $e \in E$. Then the formulas in
part~(1) specialize to:
\begin{align}
\begin{aligned}
\det (D_q(G) + xJ) = \ & \prod_{j=1}^k \left( (q-1) \det D_q(C_{r_j}) +
r_j \frac{\det D_q(C_{r_j})}{d_{r_j,1}} \right) \prod_{e \in E} (-
[\alpha_e + \alpha'_e]) \times\\
&\ \times \left[ x + (1 - (q-1)x) \left( \sum_{j=1}^{k}
\frac{d_{r_j,1}}{(q-1) d_{r_j,1} + r_j} + \sum_{e \in E} \frac{[\alpha_e]
[\alpha'_e]}{[\alpha_e + \alpha'_e]} \right) \right].
\end{aligned}
\end{align}

\item Another special case: suppose the strong blocks of $G$ are cliques
of sizes $p_1, \dots, p_m$. Then,
\begin{equation}\label{Ehypertrees}
\det( D_q(G) + xJ ) = (-1)^{\sum_i (p_i - 1)} \prod_{i=1}^m (p_i q - q +
1) \left[ x + (1-(q-1)x) \sum_{i=1}^p \frac{p_i-1}{p_i q - q + 1}
\right],
\end{equation}
where the denominators are again placeholders, cancelled by factors in
the numerators.
\end{enumerate}
\end{theorem}

\noindent We omit the proof as it involves straightforward computations
using the results shown above.

As mentioned above, Theorem~\ref{Ppolycyclic} subsumes the corresponding
results in~\cite{BKN, bicyclic, cycle-clique, LSZ, S-hypertrees}. For
instance,~\eqref{Ehypertrees} is a twofold generalization of a result
in~\cite{S-hypertrees}, where it was shown in the special case $p_i = 3\
\forall i, \ x=0$.
As another special case, the determinant of the classical additive
distance matrix over such a cycle-clique graph is now stated:

\begin{cor}\label{Cspecial}
In the setting of Theorem~\ref{Ppolycyclic}, specialize to
$q=1$, $\alpha_i = \alpha'_i = 1\ \forall i \in [m]$,
$\beta_j = 1\ \forall j \in [r]$.
Then $\det (D(G) + xJ)$ vanishes if any cycle (i.e.~$r_j$) is even, else
\begin{equation}\label{Ecycle-clique}
\det (D(G) + xJ) = \prod_{j=1}^k r_j \prod_{i=1}^m ((-1)^{p_i - 1} p_i)
\left( x + \sum_{i=1}^m \frac{p_i-1}{p_i} + \sum_{j=1}^k \frac{1}{r_j}
\lfloor r_j / 2 \rfloor \; \lceil r_j / 2 \rceil \right).
\end{equation}
\end{cor}

In the further special case $x=0$, this formula subsumes various results
in the literature -- see prior to Corollary~\ref{Cspecial}. Note that it
is another sample special case of Theorem~\ref{Ppolycyclic}(1), and again
follows from Lemmas~\ref{Lcycle} and~\ref{Lclique} and
Proposition~\ref{Pghh2ghh}. The proof of~\eqref{Ecycle-clique} uses the
fact that $\det D(C_r)$ is zero if $r$ is even, and equals $d_{r,1} =
\lfloor r / 2 \rfloor \; \lceil r / 2 \rceil$ otherwise (see
e.g.~\cite[Theorem 3.4]{BKN}).
%}}}

%{{{1 Section 6: Inverse, determinant, cofactor-sum of the most general distance matrix
\section{Inverse, determinant, cofactor-sum of the general distance
matrix: Theorems~\ref{Tmaster},~\ref{Tinverse}}\label{Sinverse}

We now return to trees. A natural question -- pursued for all variants of
$\dt$ studied to date -- is to compute $\dt^{-1}$. This was first carried
out by Graham and Lov\'asz in~\cite{GL}; in this case and subsequent
variants, $\dt^{-1}$ is usually a rank-one update of a related
``Laplacian'' matrix or a $q$-version thereof.
Later in~\cite{BLP1}, Bapat--Lal--Pati worked with $m_e = m'_e = q$ and
$a_e = a'_e\ \forall e \in E$, and wrote:
``\textit{\dots it appears that such a formula [for $\dt^{-1}$] will be
very complicated and we leave it as an open problem.}''

We now prove Theorem~\ref{Tinverse}, which provides a closed-form
expression for $\dt^{-1}$ in our general framework, hence resolves the
open question in~\cite{BLP1} in greater generality and also extends the
Graham--Lov\'asz result. As a by-product of our proof, we will obtain
closed-form expressions for $\det(\dt), \cof(\dt)$, and will then prove
the rest of Theorem~\ref{Tmaster}.

We begin by explaining the matrix $C_\mathcal{T}$ in
Theorem~\ref{Tinverse}, via some notation:

\begin{definition}
Setting as in Theorem~\ref{Tmaster}. In other words, $T = (V,E)$ is a
tree on $n$ nodes with edge-data $\mathcal{T} = \{ (a_e = a'_e, m_e,
m'_e) : e \in E \}$.
\begin{enumerate}
\item Define the scalar
\begin{equation}
\at := \sum_{e \in E} \frac{a_e (m_e - 1)(m'_e - 1)}{m_e m'_e - 1}.
\end{equation}
In working with $\at$ in the proof below, we will assume that $m_e m'_e -
1$ is invertible in $R$, as is (the numerator of the rational function)
$\at$, via a Zariski density argument.

\item Given adjacent nodes $i \sim j$, define $T_{i \to j}$ to be the
sub-tree induced on $i, j$, and all nodes $v \in V$ such that the path
from $i$ to $v$ passes through $j$. Note this partitions the edge-set:
\[
E = \bigsqcup_{j : j \sim i}  E(T_{i \to j}), \qquad \forall i \in V.
\]

\item Finally, define for $i \in V$ the scalar
\begin{equation}
\beta_i := \frac{1}{\at} \sum_{j : j \sim i} \frac{1}{a_{ij}} \sum_{e \in
E(T_{i \to j})} \frac{a_e (m_e - 1)(m'_e - 1)}{m_e m'_e - 1},
\end{equation}
where we assume by Zariski density that $a_e, m_e m'_e - 1$ are
invertible for all $e$, as is $\at$.
\end{enumerate}
\end{definition}

With this notation in hand, we finally define $C_\mathcal{T}$ to have
entries
\begin{equation}\label{EBmatrix}
(C_\mathcal{T})_{ij} := \begin{cases}
\beta_i, & \text{if } j=i;\\
\beta_i - \frac{1}{a_{ik}}, \qquad & \text{if } j \neq i, \ j \in T_{i
\to k}.
\end{cases}
\end{equation}

\begin{remark}\label{Rbeta}
If $i \in V$ is pendant with neighbor $p(i)$, then $\beta_i =
\frac{1}{a_{i,p(i)}}$ and $C_\mathcal{T}$ has $i$th row
$\frac{1}{a_{i,p(i)}} {\bf e}_i^T$.
\end{remark}

Now we are ready to compute the inverse of $\dt$, as asserted in
Theorem~\ref{Tinverse}:
\[
\dt^{-1} = \frac{1}{\at} \tauout \tauin^T - L_\mathcal{T} + C_\mathcal{T}
\diag(\tauin),
\]
where the vectors $\tauin, \tauout \in R^V$ and the Laplacian matrix
$L_\mathcal{T} \in R^{V \times V}$ were defined in
Theorem~\ref{Tinverse}. Note for this Laplacian that ${\bf e}^T
L_\mathcal{T} = 0$.

\begin{proof}[Proof of Theorem~\ref{Tinverse}]
This proof is self-contained, modulo one fact which was shown in the
``multiplicative'' special case above: the nonvanishing of $\det(\dt)$.
Hence we may assume via Zariski density that $\det(\dt) \in R^\times$. In
particular, we work throughout this proof over the ring $R = R_0 :=
\mathbb{Q}( \{ a_e, m_e, m'_e : e \in E \})$, and then observe that all
of the results proved hold (by Zariski density) over $\Z[ \{ a_e, m_e,
m'_e : e \in E \}]$, hence in any unital commutative ring by
specialization.

In addition to $\det(\dt)$, we also use via Zariski density that the
quantities
$a_e, \ m_e m'_e - 1, \ m_e - 1, \ m'_e - 1, \at$
are invertible in $R = R_0$. We use the following notation without
further reference: for nodes $i \neq j$ let $m_{ij}$ denote the product
of the edgeweights over the unique directed path in $T$ from $i$ to $j$.
Also set $m_{ii} = 1$. Let the vector ${\bf m}_{\bullet \to j} \in R^{V
\times 1}$ have $i$th coordinate $m_{ij}$; here $i,j \in V$.

We now prove the result; the proof is split into steps for ease of
exposition.\medskip

\noindent \textbf{Step 1.}
We begin by showing the following identities:
\begin{align}
\tauin^T {\bf m}_{\bullet \to l} = &\ 1, \qquad \forall l \in V;
\label{Eprojesh1}\\
\tauin^T {\bf e} = {\bf e}^T \tauout = &\ 1 - \sum_{e \in E} \frac{(m_e -
1)(m'_e - 1)}{m_e m'_e - 1}; \label{Eprojesh2}\\
{\bf e}^T C_\mathcal{T} = &\ \frac{1}{\at} \sum_{e \in E} \frac{(m_e -
1)(m'_e - 1)}{m_e m'_e - 1} \cdot {\bf e}^T; \label{Eprojesh3}
\end{align}

To prove~\eqref{Eprojesh1}, first convert the sum over nodes into a sum
over edges. For each edge $e = \{ i, j \} \in E$, if $j$ lies on the path
between $i$ and $l$, then the terms in $(\tauin^T - {\bf e}^T) {\bf
m}_{\bullet \to l}$ on the left-hand side of~\eqref{Eprojesh1}
corresponding to $e = \{ i,j \}$ contribute precisely
\[
(1 - m_e m'_e)^{-1} \left( m_{ji} (m_{ij} - 1) m_{il} + m_{ij} (m_{ji} -
1) m_{jl} \right) = -m_{il},
\]
where the positions of $i,j,l$ imply: $m_{il} = m_{ij} m_{jl}$.
Therefore,
\begin{equation}\label{Ename}
\tauin^T {\bf m}_{\bullet \to l} =
\sum_{i \in V} m_{il} - \sum_{e \in E} m_{i_l(e) l},
\end{equation}
where $i_l(e)$ is the vertex of $e \in E$ that is farther away from $l$.
Now the map $i_l : E \to V \setminus \{ l \}$ is a bijection, since $T$
is a tree. Hence the above computation yields precisely $m_{ll} = 1$,
proving~\eqref{Eprojesh1}.

We next show~\eqref{Eprojesh2}. The first equality is easily shown by
converting sums over nodes to ones over edges. For the second equality,
by summing the components of $\tauout$ and again converting the sum over
nodes to one over edges, one obtains
\[
1 + \sum_{e \in E} \left( 1 - \frac{m_e (m'_e - 1)}{m_e m'_e - 1} -
\frac{m'_e (m_e - 1)}{m_e m'_e - 1} \right) = 1 + \sum_{e \in E}
\frac{-(m_e - 1)(m'_e - 1)}{m_e m'_e - 1},
\]
as desired.
We next turn to~\eqref{Eprojesh3}. For each $j \in V$, note that
\[
({\bf e}^T C_\mathcal{T})_j = \sum_{i \in V} \beta_i - \sum_{i \in V
\setminus \{ j \}} \frac{1}{a_{ik}},
\]
where the $k$ in the final summand is such that $j \in T_{i \to k}$ -- in
other words, $k$ is the neighbor of $i$ that is closest to $j$. Now by
the reasoning following~\eqref{Ename}, the latter sum can be converted
into a sum over edges to yield precisely $\displaystyle \sum_{e \in E}
\frac{1}{a_e}$. As for the former sum, first write for convenience:
\[
\varphi_e := \frac{a_e (m_e - 1)(m'_e - 1)}{m_e m'_e - 1};
\]
thus e.g.\ $\displaystyle \at = \sum_{e \in E} \varphi_e$. Now convert
$\sum_i \beta_i$, which is a sum over nodes, into one over edges:
\[
\sum_i \beta_i = \frac{1}{\at} \sum_i \sum_{k \sim i} \frac{1}{a_{ik}}
\sum_{e \in E(T_{i \to k})} \varphi_e
= \frac{1}{\at} \sum_{e = \{ i, j \} \in E} \frac{1}{a_e} \left(
\sum_{f \in E(T_{i \to j})} \varphi_f + 
\sum_{f \in E(T_{j \to i})} \varphi_f \right).
\]
The summand in the outer sum on the right is $\displaystyle
\frac{1}{a_e} \left( \at + \varphi_e \right)$. Combining these
shows~\eqref{Eprojesh3}:
\[
({\bf e}^T C_\mathcal{T})_j = \frac{1}{\at} \sum_{e \in E} \frac{\at +
\varphi_e}{a_e} - \sum_{e \in E} \frac{1}{a_e} = \frac{1}{\at} \sum_{e
\in E} \frac{\varphi_e}{a_e},
\]

\noindent \textbf{Step 2.}
We claim the following identity also holds for each fixed node $l \in V$:
\begin{equation}\label{Eprojesh4}
C_\mathcal{T} {\bf e}_l = (-L_\mathcal{T} + C_\mathcal{T} \diag(\tauin))
{\bf m}_{\bullet \to l}, \ \ l \in V.
\end{equation}

We will show the equality of the $i$th components, for all $i,l \in V$.
First note that the $i$th component of the left-hand side is $\beta_i$
for $i = l$, else it is $\beta_i - \frac{1}{a_{ik_0}}$, where $k_0 \sim
i$ is between $i$ and $l$.

We claim the $i$th component of the right-hand side is the same.
Let ${\bf r}_i$ denote the $i$th row of $-L_\mathcal{T} + C_\mathcal{T}
\diag(\tauin)$. To show ${\bf r}_i \cdot {\bf m}_{\bullet \to l}$
equals $\beta_i$ or $\beta_i - \frac{1}{a_{i k_0}}$, we will partition
the dot-product
\begin{equation}\label{Edot}
{\bf r}_i \cdot {\bf m}_{\bullet \to l} = \sum_{v \in V} ({\bf r}_i)_v
m_{vl}
\end{equation}
into sub-summations over $V(T_{i \to k}) \setminus \{ i \}$, running over
$k \sim i$ -- and to each such sum, we will add a component
of the $i$th summand of~\eqref{Edot}. (Recall that the sets $V(T_{i \to
k}) \setminus \{ i \}$ partition $V \setminus \{ i \}$.)

First write out the $i$th summand of~\eqref{Edot}:
\begin{equation}\label{Eithpart}
\sum_{k \sim i} \frac{-m_{ki}}{a_{ik} (m_{ik} m_{ki} - 1)} \cdot m_{il} +
(\tauin)_i \beta_i m_{il}.
\end{equation}

Consider only the sum of the terms in~\eqref{Edot} that contain the
$\beta_i$'s that come from $C_\mathcal{T}$ (i.e.~from $({\bf r}_i)_v$).
Each entry of the $i$th row of $C_\mathcal{T}$ contains a $\beta_i$ term,
and by~\eqref{Eprojesh1} these contribute
\[
\beta_i \,  {\bf e}^T \diag(\tauin) \cdot {\bf m}_{\bullet \to l} =
\beta_i \cdot \tauin^T {\bf m}_{\bullet \to l} = \beta_i.
\]

With these preliminaries, we proceed. Given a node $k \sim i$, let $G_k$
denote the sub-tree induced on $T_{i \to k} \setminus \{ i \}$; thus
$V(G_k) = V(T_{i \to k}) \setminus \{ i \}$. Using the above
observations, the sum~\eqref{Edot} equals
\begin{equation}\label{Eprojesh5}
{\bf r}_i \cdot {\bf m}_{\bullet \to l} = \sum_{v \in V} ({\bf r}_i)_v
m_{vl} = \beta_i + \sum_{k \sim i} \Psi_k,
\end{equation}
where
\begin{equation}
\Psi_k := \frac{-m_{ki}}{a_{ik} (m_{ik} m_{ki} - 1)} \cdot m_{il} +
\frac{m_{ik}}{a_{ik} (m_{ik} m_{ki} - 1)} \cdot m_{kl} +
\frac{-1}{a_{ik}} \sum_{v \in V(G_k)} (\tauin)_v m_{vl}.
\end{equation}

We now claim that $\Psi_k = 0$ if $l \not\in V(G_k)$. Indeed, $m_{vl} =
m_{vk} m_{kl}$ for $v \in V(G_k)$, and taking $m_{kl}$ common outside of
the latter sum yields an expression analogous to~\eqref{Eprojesh1}, but
for $G_k$. There is but one extra component in this sum in the $v=k$
term:
\begin{align*}
\Psi_k = &\ \frac{-m_{ki}}{a_{ik} (m_{ik} m_{ki} - 1)} \cdot m_{il} +
\frac{m_{ik}}{a_{ik} (m_{ik} m_{ki} - 1)} \cdot m_{kl} +
\frac{-m_{kl}}{a_{ik}} \left( (\tauin^{(G_k)})^T {\bf m}^{(G_k)}_{\bullet
\to k} - \frac{m_{ik} (m_{ki}-1)}{m_{ik} m_{ki} - 1} \right)\\
= &\ \frac{-m_{ki}}{a_{ik} (m_{ik} m_{ki} - 1)} \cdot m_{il} +
\frac{m_{kl}}{a_{ik} (m_{ik} m_{ki} - 1)} \left( m_{ik} - (m_{ik} m_{ki}
- 1) + m_{ik} (m_{ki} - 1) \right)\\
= &\ \frac{-m_{ki}}{a_{ik} (m_{ik} m_{ki} - 1)} \cdot m_{il} +
\frac{m_{kl}}{a_{ik} (m_{ik} m_{ki} - 1)}.
\end{align*}
(Here, $\tauin^{(G_k)}, {\bf m}^{(G_k)}_{\bullet \to k}$ are clear from
context.)
But this vanishes as $m_{kl} = m_{ki} m_{il}$. Hence $\Psi_k = 0$.

This already concludes the verification if $i=l$, by~\eqref{Eprojesh5}.
Thus we suppose now that $i \neq l$, in which case there is a unique $k_0
\sim i$ such that $l \in T_{i \to k_0}$. Again by~\eqref{Eprojesh5}, it
suffices to show that $\Psi_{k_0} = -1/a_{ik_0}$. But this essentially is
a repetition of the above computation, where now we \textit{do not} take
an $m_{kl}$ common. We leave the relevant details to the reader.\medskip

\noindent \textbf{Step 3.}
We now prove our formula for the inverse of $\dt$ by induction on $n$ (or
on $|E|$). The base case of $n=2$ is a straightforward verification. For
the induction step, we assume the formula for the inverse of $D = \dt$
for a tree on $k$ nodes:
\[
D^{-1} = \frac{1}{\alpha} \tauout \tauin^T - L + C \diag(\tauin).
\]
Here and below, $\alpha = \at$, $L = L_\mathcal{T}$, and $C =
C_\mathcal{T}$. Now using the
identities~\eqref{Eprojesh1}--\eqref{Eprojesh4} from Step~1,
\begin{equation}\label{Eprojesh6}
{\bf e}^T D^{-1} = \frac{1}{\alpha} \tauin^T, \qquad
D^{-1} {\bf m}_{\bullet \to k} = \frac{1}{\alpha} \tauout + C {\bf e}_k.
\end{equation}

Write $\dt$ over $k+1$ vertices as a block-matrix, assuming that node
$k+1$ is pendant and $k+1 \sim k$:
\begin{align}\label{EDbar}
\begin{aligned}
\overline{D} := &\ D_{\mathcal{T}_{k+1}} = \begin{pmatrix} D & {\bf u} \\
{\bf v}^T  & 0 \end{pmatrix},\\
\text{where } {\bf u} := &\ D {\bf e}_k + a_{k,k+1} (m_{k,k+1} - 1) {\bf
m}_{\bullet \to k},\\
\text{and } {\bf v}^T := &\ a_{k,k+1} (m_{k+1,k} - 1) {\bf e}(k)^T +
m_{k+1,k} {\bf e}_k^T D.
\end{aligned}
\end{align}

We write $\overline{D}, \overline{\alpha}, \overline{L},
\overline{\tauin}, \overline{\tauout}$ for the larger distance matrix
(i.e., on $k+1$ nodes). Thus,
\begin{align}
\overline{L} := &\ \begin{pmatrix}
\displaystyle L + \frac{m_{k+1,k}}{a_{k,k+1} (m_{k,k+1} m_{k+1,k} - 1)}
{\bf e}_k {\bf e}_k^T & \displaystyle
\frac{-m_{k,k+1}}{a_{k,k+1} (m_{k,k+1} m_{k+1,k} - 1)} {\bf e}_k\\
\displaystyle \frac{-m_{k+1,k}}{a_{k,k+1} (m_{k,k+1} m_{k+1,k} - 1)} {\bf
e}_k^T &
\displaystyle \frac{m_{k,k+1}}{a_{k,k+1} (m_{k,k+1} m_{k+1,k} - 1)}
\end{pmatrix}, \label{ELbar}\\
\overline{\tauin} := &\ \begin{pmatrix} \tauin - \displaystyle
\frac{m_{k+1,k} (m_{k,k+1}-1)}{m_{k,k+1} m_{k+1,k} - 1} {\bf e}_k \\
\displaystyle \frac{m_{k,k+1}-1}{m_{k,k+1} m_{k+1,k} - 1} \end{pmatrix},
\label{Etauinbar}\\
\overline{\tauout} := &\ \begin{pmatrix} \tauout - \displaystyle
\frac{m_{k,k+1} (m_{k+1,k}-1)}{m_{k,k+1} m_{k+1,k} - 1} {\bf e}_k \\
\displaystyle \frac{m_{k+1,k}-1}{m_{k,k+1} m_{k+1,k} - 1} \end{pmatrix},
\label{Etauoutbar}\\
\overline{\alpha} := &\ \alpha + \frac{a_{k,k+1} (m_{k,k+1}-1)
(m_{k+1,k}-1)}{m_{k,k+1} m_{k+1,k} - 1}.
\end{align}

Now recall the formula for the inverse of a block matrix -- via Schur
complements -- and apply this to the distance matrix~\eqref{EDbar}. The
following assumes by Zariski density that $D, \overline{D}$ are
invertible:
\begin{equation}\label{EDinv}
\overline{D}^{-1} = \begin{pmatrix}
D^{-1} + \psi^{-1} D^{-1} {\bf u} {\bf v}^T D^{-1} & - \psi^{-1} D^{-1}
{\bf u} \\
- \psi^{-1} {\bf v}^T D^{-1} & \psi^{-1}
\end{pmatrix},
\qquad \text{where } \psi := -{\bf v}^T D^{-1} {\bf u} = \frac{\det
\overline{D}}{\det D}.
\end{equation}

In this step we compute $\psi$ and show that $\psi^{-1}$ equals the
corresponding $(2,2)$-block of our claimed formula for the inverse:
\begin{equation}\label{EDinv-formula}
\frac{1}{\overline{\alpha}} \, \overline{\tauout} \, \overline{\tauin}^T
- \overline{L} + \overline{C} \diag (\overline{\tauin}),
\end{equation}
where this matrix is partitioned in the same manner as~\eqref{EDinv}.
The other three components of the matrix~\eqref{EDinv} will then be
computed in subsequent steps, together with showing that they equal the
corresponding blocks of the matrix~\eqref{EDinv-formula}. This will
conclude the proof. We begin by computing:
\begin{align*}
-\psi = {\bf v}^T D^{-1} {\bf u}
= &\ (a_{k,k+1} (m_{k+1,k} - 1) {\bf e}(k)^T + m_{k+1, k} {\bf e}_k^T D)
D^{-1} (D {\bf e}_k + a_{k,k+1} (m_{k,k+1} - 1) {\bf m}^{(k)}_{\bullet
\to k})\\
= &\ a_{k,k+1} (m_{k+1, k} - 1) {\bf e}(k)^T {\bf e}_k
+ a_{k,k+1}^2 (m_{k+1, k} - 1) (m_{k,k+1} - 1) {\bf e}(k)^T \cdot
D^{-1} \cdot {\bf m}^{(k)}_{\bullet \to k}\\
&\ + m_{k+1, k} {\bf e}_k^T D {\bf e}_k + a_{k,k+1} m_{k+1, k} (m_{k,k+1}
- 1) {\bf e}_k^T \cdot {\bf m}^{(k)}_{\bullet \to k}.
\end{align*}

\noindent The third term is zero; the first and fourth terms add up to
$a_{k,k+1} (m_{k,k+1} m_{k+1,k} - 1)$. Hence,
\begin{equation}\label{Epsi}
\psi = a_{k,k+1} (1 - m_{k,k+1} m_{k+1,k}) - \frac{a_{k,k+1}^2}{\alpha}
(m_{k,k+1} - 1)(m_{k+1,k} - 1)
= \frac{a_{k,k+1} (1 - m_{k,k+1} m_{k+1,k}) \overline{\alpha}}{\alpha},
\end{equation}
where the first equality follows from~\eqref{Eprojesh6}
and~\eqref{Eprojesh1}.

As discussed following~\eqref{EDinv}, we now show that $\psi^{-1}$ equals
the $(2,2)$-block of~\eqref{EDinv-formula}, to complete this step of the
proof of Theorem~\ref{Tinverse}. The latter $(2,2)$-block equals the
scalar
\begin{align*}
&\ \frac{1}{\overline{\alpha}} 
\frac{m_{k,k+1}-1}{m_{k,k+1} m_{k+1,k} - 1}
\frac{m_{k+1,k}-1}{m_{k,k+1} m_{k+1,k} - 1} - 
\frac{m_{k,k+1}}{a_{k,k+1}(m_{k,k+1} m_{k+1,k} - 1)} +
\frac{1}{a_{k,k+1}} \frac{m_{k,k+1}-1}{m_{k,k+1} m_{k+1,k} - 1},
\end{align*}
where the final term comes from Remark~\ref{Rbeta}. Now an easy
computation simplifies this to (by~\eqref{Epsi})
\[
\frac{\alpha}{a_{k,k+1} (1 - m_{k,k+1} m_{k+1,k}) \overline{\alpha}} =
\psi^{-1}.
\]

\noindent \textbf{Step 4.}
Next, we show the $(2,1)$-blocks of the expressions~\eqref{EDinv}
and~\eqref{EDinv-formula} agree. The former is
\[
-\psi^{-1} {\bf v}^T D^{-1} = \frac{m_{k+1,k} - 1}{\overline{\alpha}
(m_{k,k+1} m_{k+1,k} - 1)} \tauin^T + \frac{\alpha \,
m_{k+1,k}}{\overline{\alpha} \, a_{k,k+1} (m_{k,k+1} m_{k+1,k} - 1)} {\bf
e}_k^T,
\]
using~\eqref{Eprojesh6} and~\eqref{Epsi}. By the definition of $\alpha,
\overline{\alpha}$, a straightforward computation shows this equals
\[
\frac{m_{k+1,k} - 1}{\overline{\alpha} \, (m_{k,k+1} m_{k+1,k} - 1)}
\left( \tauin^T - \frac{m_{k+1,k} (m_{k,k+1} - 1)}{m_{k,k+1} m_{k+1,k} -
1} {\bf e}_k^T \right) + \frac{m_{k+1,k}}{a_{k,k+1} (m_{k,k+1} m_{k+1,k}
- 1)} {\bf e}_k^T,
\]
which equals the $(2,1)$-block of~\eqref{EDinv-formula} via
Remark~\ref{Rbeta} for $i=k+1$.\medskip

We now consider the $(1,2)$-blocks in~\eqref{EDinv}
and~\eqref{EDinv-formula}. We begin with the former;
using~\eqref{Eprojesh6},
\begin{align*}
&\ - \psi^{-1} D^{-1} {\bf u}\\
= &\ \frac{m_{k,k+1} - 1}{\overline{\alpha} (m_{k,k+1} m_{k+1,k} - 1)}
\tauout
+ \frac{\alpha (m_{k,k+1} - 1)}{\overline{\alpha} (m_{k,k+1} m_{k+1,k} -
1)} C {\bf e}_k
+ \frac{\alpha}{\overline{\alpha} \, a_{k,k+1} (m_{k,k+1} m_{k+1,k} - 1)}
{\bf e}_k.
\end{align*}

\noindent Now add-and-subtract two terms to get:
\begin{align*}
= &\ \frac{m_{k,k+1} - 1}{\overline{\alpha} (m_{k,k+1} m_{k+1,k} - 1)}
\left( \tauout - \frac{m_{k,k+1} (m_{k+1,k} - 1)}{m_{k,k+1} m_{k+1,k} -
1} {\bf e}_k \right)\\
&\ + \frac{m_{k,k+1} (m_{k,k+1} - 1) (m_{k+1,k} - 1)}{\overline{\alpha}
(m_{k,k+1} m_{k+1,k} - 1)^2} {\bf e}_k +
\frac{\alpha \, m_{k,k+1}}{\overline{\alpha} \, a_{k,k+1} (m_{k,k+1}
m_{k+1,k} - 1)} {\bf e}_k\\
&\ - \frac{\alpha \, m_{k,k+1}}{\overline{\alpha} \, a_{k,k+1} (m_{k,k+1}
m_{k+1,k} - 1)} {\bf e}_k + \frac{\alpha}{\overline{\alpha} \,
a_{k,k+1} (m_{k,k+1} m_{k+1,k} - 1)} {\bf e}_k + \frac{\alpha (m_{k,k+1}
- 1)}{\overline{\alpha} (m_{k,k+1} m_{k+1,k} - 1)} C {\bf e}_k.
\end{align*}

\noindent The terms on the first line (resp.~second line) in the
right-hand side here add up to yield the $(1,2)$-block of
$\frac{1}{\overline{\alpha}} \overline{\tauout} \; \overline{\tauin}^T$
(resp.~of $-\overline{L}$). The terms in the third line add up to yield
\[
\frac{\alpha (m_{k,k+1} - 1)}{\overline{\alpha} (m_{k,k+1} m_{k+1,k} -
1)} \left( C {\bf e}_k - \frac{1}{a_{k,k+1}} {\bf e}_k \right),
\]
and a straightforward but careful calculation shows this equals the
$(1,2)$-block of $\overline{C} \diag(\overline{\tauin})$. This concludes
the verification for the $(1,2)$-blocks of~\eqref{EDinv}
and~\eqref{EDinv-formula}.\medskip

\noindent \textbf{Step 5.}
In this final step, we handle the most involved of the computations: the
equality of the $(1,1)$-blocks of~\eqref{EDinv}
and~\eqref{EDinv-formula}. As the computations are fairly involved, we
begin by outlining the strategy. We expand out the expression
\[
D^{-1} + \psi^{-1} \cdot (D^{-1} {\bf u}) \cdot ({\bf v}^T D^{-1}),
\]
substituting $D^{-1} = \frac{1}{\alpha} \tauout \tauin^T - L
+ C \diag(\tauin)$ in the first term, and computing $D^{-1} {\bf u}$,
${\bf v}^T D^{-1}$ using the intermediate identities proved above. Then
we rearrange terms to obtain expressions for the $(1,1)$-blocks of
$\frac{1}{\overline{\alpha}} \overline{\tauout} \; \overline{\tauin}^T$
and $-\overline{L}$, plus some extra terms. Finally, these extra terms
will be shown to add up to the $(1,1)$-block of $\overline{C}
\diag({\overline{\tauin}})$.

Thus, we begin with
\begin{align}
&\ D^{-1} + \psi^{-1} \cdot (D^{-1} {\bf u}) \cdot ({\bf v}^T D^{-1})
\label{E11block1}\\
= &\ \frac{1}{\alpha} \tauout \tauin^T - \frac{a_{k,k+1} (m_{k,k+1} -
1)(m_{k+1,k} - 1)}{\alpha \overline{\alpha} \, (m_{k,k+1} m_{k+1,k} - 1)}
\tauout \tauin^T - \frac{m_{k+1,k} (m_{k,k+1} - 1)}{\overline{\alpha} \,
(m_{k,k+1} m_{k+1,k} - 1)} \tauout {\bf e}_k^T \notag\\
&\ -L - \frac{\alpha \, m_{k+1,k}}{\overline{\alpha} \, a_{k,k+1}
(m_{k,k+1} m_{k+1,k} - 1)} {\bf e}_k {\bf e}_k^T \notag\\
&\ + C \diag(\tauin) - \frac{m_{k+1,k}-1}{\overline{\alpha} \, (m_{k,k+1}
m_{k+1,k} - 1)} {\bf e}_k \tauin^T - \frac{\alpha \, m_{k+1,k} (m_{k,k+1}
- 1)}{\overline{\alpha} \, (m_{k,k+1} m_{k+1,k} - 1)} C {\bf e}_k {\bf
e}_k^T \notag\\
&\ - \frac{a_{k,k+1} (m_{k,k+1} - 1) (m_{k+1,k} - 1)}{\overline{\alpha}
\, (m_{k,k+1} m_{k+1,k} - 1)} C {\bf e}_k \tauin^T.\notag
\end{align}

Now the first three terms (on the first line on the right-hand side
of Equation~\eqref{E11block1}) add up to
\begin{align}\label{E11block2}
\begin{aligned}
&\ \frac{1}{\overline{\alpha}} \left( \tauout - \frac{m_{k,k+1}
(m_{k+1,k}-1)}{m_{k,k+1} m_{k+1,k} - 1} {\bf e}_k \right) \left( \tauin -
\frac{m_{k+1,k} (m_{k,k+1}-1)}{m_{k,k+1} m_{k+1,k} - 1} {\bf e}_k
\right)^T\\
+ &\ \frac{m_{k,k+1} (m_{k+1,k} - 1)}{\overline{\alpha} \, (m_{k,k+1}
m_{k+1,k} - 1)} {\bf e}_k \tauin^T
- \frac{m_{k,k+1} m_{k+1,k} (m_{k,k+1} - 1) (m_{k+1,k} -
1)}{\overline{\alpha} \, (m_{k,k+1} m_{k+1,k} - 1)^2} {\bf e}_k {\bf
e}_k^T.
\end{aligned}
\end{align}
Similarly, the next two terms (on the second line on the right-hand side
of~\eqref{E11block1}) add up to
\begin{equation}\label{E11block3}
\left( - L - \frac{m_{k+1,k}}{a_{k,k+1} (m_{k,k+1} m_{k+1,k} - 1)} {\bf
e}_k {\bf e}_k^T \right) + \frac{m_{k+1,k} (m_{k,k+1} - 1) (m_{k+1,k} -
1)}{\overline{\alpha} \, (m_{k,k+1} m_{k+1,k} - 1)^2} {\bf e}_k {\bf
e}_k^T.
\end{equation}

\noindent Notice by~\eqref{ELbar}, \eqref{Etauinbar}, \eqref{Etauoutbar}
that the first expression in~\eqref{E11block2} (resp.~\eqref{E11block3})
equals the $(1,1)$-block of $\frac{1}{\overline{\alpha}}
\overline{\tauout} \; \overline{\tauin}^T$ (resp.~$-\overline{L}$). Thus,
it remains to show that the $(1,1)$-block of $\overline{C}
\diag(\overline{\tauin})$ equals the sum of the remaining seven terms
in~\eqref{E11block1}, \eqref{E11block2}, \eqref{E11block3}, which we now
collect together:
\begin{align}\label{E11block4}
\begin{aligned}
&\ C \diag(\tauin)
- \frac{\alpha \, m_{k+1,k} (m_{k,k+1} - 1)}{\overline{\alpha} \,
(m_{k,k+1} m_{k+1,k} - 1)} C {\bf e}_k {\bf e}_k^T
- \frac{a_{k,k+1} (m_{k,k+1} - 1) (m_{k+1,k} - 1)}{\overline{\alpha}
\, (m_{k,k+1} m_{k+1,k} - 1)} C {\bf e}_k \tauin^T\\
+ &\ \frac{(m_{k,k+1} - 1) (m_{k+1,k} - 1)}{\overline{\alpha} \, (m_{k,k+1}
m_{k+1,k} - 1)} {\bf e}_k \tauin^T
- \frac{m_{k+1,k} (m_{k,k+1} - 1)^2 (m_{k+1,k} - 1)}{\overline{\alpha} \,
(m_{k,k+1} m_{k+1,k} - 1)^2} {\bf e}_k {\bf e}_k^T.
\end{aligned}
\end{align}
In~\eqref{E11block4}, the two expressions on the last line are each
obtained by combining two of the ``remaining seven terms'' above.
Now define the vector
\begin{equation}
\tauin' := \tauin - \frac{m_{k+1,k} (m_{k,k+1} - 1)}{m_{k,k+1}
m_{k+1,k} - 1} {\bf e}_k
\end{equation}
and notice this precisely equals the $(1,1)$-block of $\overline{\tauin}$
by~\eqref{Etauinbar}.
Then the last two terms -- all on the second line -- of~\eqref{E11block4}
add up to yield
\[
\frac{(m_{k,k+1} - 1) (m_{k+1,k} - 1)}{\overline{\alpha} \, (m_{k,k+1}
m_{k+1,k} - 1)} {\bf e}_k (\tauin')^T.
\]

Similarly, the first three terms -- all on the first line --
of~\eqref{E11block4} add up to give
\[
C \diag(\tauin') - \frac{a_{k,k+1} (m_{k,k+1} -
1)(m_{k+1,k}-1)}{\overline{\alpha} (m_{k,k+1} m_{k+1,k} - 1)} C {\bf e}_k
( \tauin' )^T.
\]
Indeed, break up the second term in~\eqref{E11block4} via $\alpha =
\overline{\alpha} - \frac{a_{k,k+1} (m_{k,k+1} -
1)(m_{k+1,k}-1)}{m_{k,k+1} m_{k+1,k} - 1}$, and add these components to
the first and third terms respectively, using $\diag(\tauin) + \gamma
{\bf e}_k {\bf e}_k^T = \diag(\tauin + \gamma {\bf e}_k)$.

Now since ${\bf e}_k (\tauin')^T = {\bf e}_k {\bf e}^T \diag(\tauin')$,
the terms in~\eqref{E11block4} all add up to
\[
\left( C - \frac{a_{k,k+1} (m_{k,k+1} -
1)(m_{k+1,k}-1)}{\overline{\alpha} (m_{k,k+1} m_{k+1,k} - 1)} (C {\bf
e}_k) {\bf e}^T + \frac{(m_{k,k+1} - 1) (m_{k+1,k} -
1)}{\overline{\alpha} \, (m_{k,k+1} m_{k+1,k} - 1)} {\bf e}_k {\bf e}^T
\right) \diag(\tauin').
\]
Notice the final summand only updates the final row of the first term.
Now another careful computation shows that the preceding expression
indeed equals the $(1,1)$-block of $\overline{C}
\diag(\overline{\tauin})$. This concludes the proof of the
inverse-formula, by induction.
\end{proof}

As a by-product of the above proof, we can now show the remaining main
theorem above:

\begin{proof}[Proof of Theorem~\ref{Tmaster}]
The formula~\eqref{Emaster} (with $x=0$) for $\det(\dt)$ easily follows
by induction on $|E|$ and~\eqref{EDinv},~\eqref{Epsi}.
Moreover, the formula for $\cof(\dt) = \det(\dt) \cdot ({\bf e}^T
\dt^{-1} {\bf e})$ follows from~\eqref{Eprojesh6}, \eqref{Eprojesh2}, and
the formulas for $\dt^{-1}, \det(\dt)$. This shows~\eqref{Emult-more} for
$|I \Delta J'| = 0$. (Here, we assume $\dt$ is invertible by Zariski
density, since $\cof(\dt)$ is also a polynomial in the entries.)

Next using Cramer's rule, Theorem~\ref{Tinverse} provides a proof
of~\eqref{Emaster} for $|I \Delta J'| = 2$; we also provide an alternate
proof below. First, we show that $(\dt + xJ)_{I|J'}$ is singular if $|I
\Delta J'| > 2$. Let ${\bf d}_v^T$ denote the $J'$-truncated $v$th row
for $v \in V(G) \setminus I$. There are now two cases:\medskip

\noindent \textbf{Case 1:} First suppose $J' \setminus I$ contains an
edge $\{ j_1, j_2 \} \in E$. Then there exists a unique node $p \in T
\setminus J'$ that is closest to $j_2$. Without loss of generality,
interchange the labels $j_1, j_2$ such that $p$ is closer to $j_2$ than
$j_1$; then the path from $j_1$ to $p$ has length at least $2$. Denote
this path by
\[
j_1 \ \longleftrightarrow \ j_2 \ \longleftrightarrow \ \cdots \
\longleftrightarrow \ a \ \longleftrightarrow \ b \ \longleftrightarrow \
p.
\]
Then $a,b \in J'$ by the assumptions, and we re-set $j_1 := a, j_2 := b$,
so that $p \sim j_2$ and $j_2 \sim j_1$. Now $p$ cannot lie in $I$, else
the maximum sub-tree containing $p \not\in J'$ but not $j_2 \not\in J'$
(resp.~$j_2$ but not $p$) would completely lie in $I$ (resp.~$J'$), in
which case $V = I \cup J'$, a contradiction. We now have:
\[
{\bf d}^T_{j_1} - m_{j_1 j_2} {\bf d}^T_{j_2} = (a_{j_1 j_2} - x) (m_{j_1
j_2} - 1) {\bf e}^T, \qquad
{\bf d}^T_{j_2} - m_{j_2 p} {\bf d}^T_p = (a_{j_2 p} - x)(m_{j_2 p} -
1){\bf e}^T,
\]
and so the determinant of $\det(\dt + xJ)_{I|J'}$ vanishes, as
claimed.\medskip

\noindent \textbf{Case 2:} If the previous case does not hold, then $J'
\setminus I = J'$ contains only pendant vertices. Choose $j_1, j_2 \in J'$
with neighbors $p(j_1), p(j_2)$ respectively. Now it is clear from the
hypotheses that $p(j_l) \not\in I \cup J'$ for $l = 1,2$; and with the
same notation as in the previous case, we have
\[
{\bf d}^T_{j_l} - m_{j_l p(j_l)} {\bf d}^T_{p(j_l)} =
(a_{j_l p(j_l)} - x)(m_{j_l p(j_l)} - 1) {\bf e}^T, \qquad l = 1,2.
\]
Once again, it follows that $\det(\dt + xJ)_{I|J'}$ vanishes.\medskip

It remains to show~\eqref{Emult-more} when $|I \Delta J'| = 2$, and we
prove it by induction on $n \geq |I|+2$. The assertion is not hard to
verify for $n = 3,4$, so we will assume henceforth that $\mathcal{T}$ has
at least $n+1 \geq 5$ nodes, and that~\eqref{Emult-more} holds for all
trees with at most $n$ nodes. Without loss of generality we set $I = \{ 1
\}$ and $J' = \{ n+1 \}$, with both $1,n+1$ pendant vertices. Since every
tree on at least $3$ nodes has at least two pendant nodes and these are
necessarily non-adjacent, we also relabel the nodes such that if one
deletes the node $1$ (respectively, $n+1$) then the node $2$
(respectively, $n$) is pendant to the deleted portion of the tree. We may
also assume the nodes $2,n$ are not adjacent, since $n \geq 5$.

Let $D := (\dt + xJ)_{1|n+1}$; we compute $\det(D)$ using Dodgson
condensation~\cite{Dodgson}, which says:
\begin{equation}\label{Edodgson}
\det D \cdot \det D_{1n|1n} = \det D_{1|1} \cdot \det D_{n|n} - \det
D_{1|n} \cdot \det D_{n|1}.
\end{equation}
Note that
\begin{align*}
D_{1|1} = &\ (\dt + xJ)_{12|1(n+1)}, \qquad
D_{n|n} = (\dt + xJ)_{1(n+1)|n(n+1)}, \\
D_{1|n} = &\ (\dt + xJ)_{12|n(n+1)}, \qquad
D_{n|1} = (\dt + xJ)_{1(n+1)|1(n+1)}, \\
D_{1n|1n} = &\ (\dt + xJ)_{12(n+1)|1n(n+1)}.
\end{align*}
Since $n+1 \geq 5$, it follows by the preceding case of $|I \Delta J'| >
2$ that $\det D_{1|n} = 0$. Hence by~\eqref{Edodgson},
\begin{align*}
\det D = &\ (\det D_{1n|1n})^{-1} (\det D_{1|1}) (\det D_{n|n})\\
= &\ (\det (\dt + xJ)_{12(n+1)|1n(n+1)})^{-1} (\det (\dt +
xJ)_{12|1(n+1)}) (\det (\dt + xJ)_{1(n+1)|n(n+1)}).
\end{align*}
But all three terms on the right are computable by the induction
hypothesis; and we may assume by Zariski density (i.e.~a suitable
application of Lemma~\ref{Lzariski}) that the factors of $\det D_{1n|1n}$
and all $a_e$ and $m_e m'_e - 1$ are invertible, namely:
\[
a_e, \quad m_e m'_e - 1, \ (e \in E); \qquad a_{p(n),n} - x, \quad
m_{p(2),2} - 1, \quad m_{p(n),n} - 1.
\]
Now the induction step follows by a straightforward cancellation,
completing the proof.
\end{proof}

\begin{remark}\label{Rnonpendant}
The above proof for $|I \Delta J'| > 2$ also goes through verbatim for
arbitrary $I$ of size $|J'|$, which do not contain the two nodes $a, b,
p$ in case 1, or $j_1, j_2, p(j_1), p(j_2)$ in case 2. Alternately, one
can work with the transpose of $\dt + xJ$, and hence with $J'$ as
specified but more general $I$.
\end{remark}

\begin{remark}[New technique: Zariski density arguments]\label{Rzariski}
We now make some remarks on our -- to our knowledge, novel -- use of
Zariski density in this paper. There are at least three advantages:
\begin{enumerate}
\item As cited in the introduction, various papers work with
e.g.~$q$-distance matrices, under the assumption that $q$ is a real
number and $q \neq \pm 1$; or that $q$ is a parameter (see
e.g.~\cite{BLP1,BR,YY2}). Zariski density enables working over any unital
commutative ring, and immediately eliminates such restrictions.

\item Using Zariski density allows one to assume nonzero expressions such
as $\det(\dt)$ or $1 - m_e m'_e$ or $1 - m_e$ to then be invertible,
hence $\dt$. This provides stronger tools to prove results.

\item Zariski density can help clarify arguments. E.g.~in~\cite{YY1},
the authors tackle the original case $a_e = a'_e = 1$ and $m_e = m'_e =
q, \ q \to 1$, and show that $(\det D_T)^2 = - (\det D_T) |E|
(-2)^{|E|-1}$. Here one cannot \textit{a priori} cancel $\det D_T$,
unless one assumes somehow that $\det D_T \not\equiv 0$. In our case,
this can be done using Zariski density, because specializing to $a_e =
a'_e = 1\ \forall e$, it follows by Theorem~\ref{Tghhmult} that
$\det(\dt) \not\equiv 0$. This point seems not to have been made
in~\cite{YY1}, where the authors cancel $\det D_T$ in the above equation
-- in effect showing in the special case $x=0$, $a_e = a'_e = 1$, and
$m_e = m'_e = 1$, $q \to 1$ that $\det(\dt) = -|E| (-2)^{|E|-1}$ assuming
$\det(\dt) \neq 0$.\footnote{The authors have since mentioned to us
(personal communication) that they prove $\det(\dt) \neq 0$ in other
papers and hence can cancel it away. These other papers are not cited in
the proof in~\cite{YY1}.} Zariski density helps fill such gaps.\medskip
\end{enumerate}
Thus, we hope that the present work leads to Zariski density being used
in this area -- not just for distance matrices -- and also leads to the
removal of various superfluous/unnecessary mathematical restrictions, as
well as enabling one to use stronger tools, e.g.~the invertibility of
various matrices or the nonvanishing of certain polynomials.
\end{remark}

We end this section by briefly discussing instances of how
Theorem~\ref{Tinverse} specializes to known formulas in the
literature~\cite{BKN}--\cite{BS2}, \cite{GL,ZD1,ZD2}. First, for every
$q$-distance matrix, i.e.~with
$a_e = \frac{1}{q-1}$, $m_e = q^{\alpha_e}$, $m'_e = q^{\alpha'_e}$,
one verifies that
\[
C = (q-1) {\rm Id}_V, \qquad \at = \sum_{e \in E} \frac{ [\alpha_e]
[\alpha'_e]}{[\alpha_e + \alpha'_e]}.
\]
In particular, if one specializes to $q=1$ then $C=0$; if moreover the
tree is unweighted, we get
\[
\at = \frac{|E|}{2}, \qquad \tauin = \tauout = {\bf e} - {\bf d}/2,
\]
where ${\bf d}$ is the vector of node-degrees. These are precisely the
expressions that appear in the Graham--Lov\'asz formula for $\dt^{-1}$ in
the original unweighted and undirected setting~\cite{GL}.

A parallel setting involves the product distance matrix $D^*_T = \dt + J$
studied in~\cite{ZD1}. Here $a_e = 1\ \forall e$. Now one verifies that
$C = {\rm Id}_V$; moreover, $(D^*_T)^{-1}$ can be computed alternately
(to the argument in Section~\ref{Sghhmult}) by using the
Sherman--Morrison formula for $(\dt + J)^{-1}$. Carrying out the
computations using the formula for $\dt^{-1}$ and the identities shown
above yields precisely:
\[
(D^*_T)^{-1} = - L + \diag(\tauin).
\]
This specializes to several of the formulas in the literature cited in
the above discussion. Notice also that when $a_e = 1\ \forall e$, we have
$\displaystyle (D^*_T)^{-1} {\bf e} = \tauout, \ {\bf e}^T (D^*_T)^{-1} =
\tauin^T$;
this in particular reveals the presence of $\tauout, \tauin^T$ in a
formula for $\dt^{-1} = (D^*_T - J)^{-1}$, when going the reverse way via
the Sherman--Morrison formula.

\begin{remark}
Theorem~\ref{Tinverse} in particular answers an open question of
Bapat--Lal--Pati~\cite{BLP1}, where they ask for the explicit form of
$\dt^{-1}$ in the special case
$a_e = \frac{w_e}{q-1}, \ m_e = m'_e = q, \ e \in E$
with $q \neq \pm 1, w_e \neq 0$ scalars. For completeness we spell out
the specialization of Theorem~\ref{Tinverse} to this setting, assuming
all denominators below are invertible:
\[
\dt^{-1} = \frac{q+1}{\sum_{e \in E} w_e} \boldsymbol{\tau}
\boldsymbol{\tau}^T - \frac{q}{q+1} L_w + C_\mathcal{T}
\diag(\boldsymbol{\tau}),
\]
where $\boldsymbol{\tau} := {\bf e} - \frac{q}{q+1} {\bf d}$ for ${\bf
d}$ the vector of node-degrees, $L_w$ is the (symmetric) weighted
Laplacian matrix of Bapat--Kirkland--Neumann~\cite{BKN} given by
\[
(L_w)_{ij} = \begin{cases}
\frac{-1}{w_{ij}}, & \text{if } i \sim j;\\
\sum_{k \sim i} \frac{1}{w_{ik}}, \qquad & \text{if } i=j;\\
0, & \text{otherwise};
\end{cases}
\]
and $C_\mathcal{T}$ is the $V \times V$ matrix with entries given by
\[
(C_\mathcal{T})_{ij} := \beta_i - {\bf 1}_{i \neq j} \cdot
\frac{q-1}{w_{ik}} \quad \text{where} \quad j \in T_{i \to k} \quad
\text{and} \quad \beta_i := \frac{q-1}{\sum_{e \in E} w_e} \sum_{k : k
\sim i} \frac{1}{w_{ik}} \sum_{e \in E(T_{i \to k})} w_e.
\]
\end{remark}
%}}}

%{{{1 Section 7: A novel, third invariant for trees, and its Graham--Hoffman--Hosoya Theorem D
\section{A novel, third invariant for trees, and its
Graham--Hoffman--Hosoya Theorem~\ref{Tghh3}}\label{Sthird}

In this paper, we have studied three variants of the distance matrix for
trees:
(a)~the most general version $\dt$ with entries given
by~\eqref{Eedgeweight};
(b)~the ``product'' distance matrix $D^*_G$; and
(c)~the $q$-matrix $D_q(G)$ (and its $q=1$ specialization, $D_1(G)$).
In Section~\ref{Sghhmult} we stated and proved Graham--Hoffman--Hosoya
type identities in settings~(b) and~(c) -- see Proposition~\ref{Pghh2ghh}
and the preceding equations.
The latter of these identities specialized to the classical
Graham--Hoffman--Hosoya identities~\eqref{Eghh}.

It is natural to ask if there exist similar identities in the ``most
general'' setting~(a) of the present paper -- and also whether or not
these specialize to the original results~\eqref{Eghh} of~\cite{GHH}. In
this final section, we affirmatively answer both questions, and in
particular, provide a third proof of the formula~\eqref{Emaster} for
$\det(\dt), \cof(\dt)$. Our new identities below will use $\det(\cdot)$
but not $\cof(\cdot)$, and in its place we now introduce a novel, third
invariant:

\begin{definition}\label{Dthird}
Suppose $G$ is a finite directed, strongly connected graph with node set
$V$, a distinguished cut-vertex $v_0 \in V$, and $R$-valued maps $d, m :
V \times V \to R$ that satisfy:
\begin{equation}\label{Edistance}
d(v,w) = d(v,v_0) + m(v,v_0) d(v_0, w), \quad d(v_0,v_0) = 0,
\end{equation}
whenever $v,w$ lie in adjacent strong blocks, both containing $v_0$.
Given a subgraph $G'$ induced on the subset of nodes $V'$ which contains
$v_0$, write
\[
D_{G'} := (d(v,w))_{v,w \in V'} = \begin{pmatrix} D|_{V' \setminus \{ v_0
\}} & {\bf u}_1 \\ {\bf w}_1^T & 0 \end{pmatrix},
\]
by relabelling the nodes, and define the invariant
\begin{equation}\label{Ekappa}
\kappa(D_{G'}, v_0) := \det \left( D|_{V' \setminus \{ v_0 \}} - {\bf
u}_1 \, {\bf e}^T - {\bf m}(V' \setminus \{ v_0 \}, v_0) {\bf w}_1^T
\right).
\end{equation}
Note here that ${\bf u}_1 = {\bf d}(V' \setminus \{ v_0 \}, v_0)$ and
${\bf w}_1 = {\bf d}(v_0, V' \setminus \{ v_0 \})$.
\end{definition}

\begin{remark}
For arbitrary graphs $G$, the notion of distance matrix in
Definition~\ref{Dthird} is a general one. When one works with $G = T$ a
tree, this data is precisely that of our setting in~\eqref{Eedgeweight},
via:
$d_e \leftrightarrow a_e(m_e - 1), \
d'_e \leftrightarrow a_e(m'_e - 1)$,
since $m_e, m'_e$ are parameters (hence unequal in general to $1$).
In this case, one has~\eqref{Edistance} for any two nodes $v,w$ and any
intermediate node $v_0$.
\end{remark}

With this terminology in hand, we state three results on the invariant
$\kappa$, deferring the proofs to a later subsection. These results
compute $\kappa(D_G, v_0)$ for graphs $G$ with a cut-vertex $v_0$, as
well as $\kappa(\dt)$ for arbitrary trees. We will then mention a few
corollaries, and end with Example~\ref{Exthird} which shows that our
results are, once again, ``best possible'' in a sense.

Our first -- and final main -- result presents GHH-type identities for
the three invariants:

\begin{utheorem}\label{Tghh3}
Notation as in Definition~\ref{Dthird}. Let $G_1, \dots, G_k$ be
strongly connected subgraphs of $G$ containing $v_0 \in V$ such that the
sets $V(G_j) \setminus \{ v_0 \}$ are pairwise disjoint. Then,
\begin{align}
\begin{aligned}
\kappa(D_G, v_0) = &\ \prod_{j=1}^k \kappa(D_{G_j}, v_0),\\
\det(D_G) = &\ \sum_{j=1}^k \det(D_{G_j}) \prod_{i \neq j}
\kappa(D_{G_i}, v_0),\\
\cof(D_G) = &\ \kappa(D_G,v_0) + \sum_{j=1}^k (\cof(D_{G_j}) -
\kappa(D_{G_j}, v_0)) \prod_{i \neq j} \kappa(D_{G_i}, v_0).
\end{aligned}
\end{align}
In other words, if $\kappa(D_{G_j},v_0) \in R^\times\ \forall j$, then
\[
\frac{\det(D_G)}{\kappa(D_G, v_0)} = \sum_{j=1}^k
\frac{\det(D_{G_j})}{\kappa(D_{G_j}, v_0)}, \qquad
\frac{\cof(D_G)}{\kappa(D_G, v_0)} - 1 = \sum_{j=1}^k
\left( \frac{\cof(D_{G_j})}{\kappa(D_{G_j}, v_0)} - 1 \right).
\]
\end{utheorem}

\noindent Notice these formulas are similar in form to~\eqref{Eghh}.

Next, we show that $\kappa(\cdot)$ is indeed an invariant for trees, as
stated above:

\begin{theorem}\label{Tthird}
Suppose $\mathcal{T} = \{ (a_e, m_e, m'_e) : e \in E \}$ comprise the
edge-data of $T$, as in Theorem~\ref{Tmaster}. In this case we define
$\kappa(\dt, v_0)$ for any vertex $v_0 \in V$, by the same formula as
in~\eqref{Ekappa}. Then $\kappa(\dt, v_0)$ depends on neither the choice
of (cut or pendant) node $v_0 \in V$, nor the tree-structure of $\dt$. It
only depends on the edge-data, as follows:
\begin{equation}
\kappa(\dt, v_0) = \prod_{e \in E} a_e (1 - m_e m'_e).
\end{equation}
In particular, $\kappa(\dt, v_0)$ is multiplicative over subgraphs cut by
$v_0$.
When $\kappa(D_e)$ is invertible for all edges $e$, then one has,
parallel to Theorem~\ref{Tghh3}:
\begin{equation}\label{Eqthird}
\frac{\det(\dt)}{\kappa(\dt)} = \sum_{e \in E}
\frac{\det(D_e)}{\kappa(D_e)}, \qquad \qquad
\frac{\cof(\dt)}{\kappa(\dt)} - 1 = \sum_{e \in E}
\left(\frac{\cof(D_e)}{\kappa(D_e)} - 1 \right).
\end{equation}
\end{theorem}

\begin{remark}
Akin to Remark~\ref{Rkaushal}, while one can use Theorem~\ref{Tthird} to
deduce the formulas for $\det(\dt), \cof(\dt)$ in Theorem~\ref{Tmaster},
once again the formulas for $\det(\cdot), \cof(\cdot)$ of the more
general submatrices $(\dt + xJ)_{I|J'}$ do \textit{not} follow from these
results.
\end{remark}

\begin{remark}
Given Theorem~\ref{Tthird}, we write $\kappa(\dt, v_0)$ as (the
invariant) $\kappa(\dt)$ henceforth.
\end{remark}

From Theorem~\ref{Tthird} it is possible to deduce the formulas for
$\det(\dt)$ and $\cof(\dt)$ as in~\eqref{Emaster} (or
Theorem~\ref{Tmaster} with $I = J' = \emptyset$). Our third result here
shows that the converse is also true:

\begin{prop}\label{Pthird}
Notation as in Theorem~\ref{Tthird}. The following can be deduced from
each other:
\begin{enumerate}
\item For all such trees and all nodes $v_0 \in V$,
$\det(\dt) = \displaystyle \prod_{e \in E} (a_e (1 - m_e m'_e)) \sum_{e
\in E} \frac{\det(D_e)}{\kappa(D_e)}$, where the denominators are
understood to be placeholders to cancel with a factor outside the sum.

\item For all such trees and all nodes $v_0 \in V$,
\[
\cof(\dt) = \prod_{e \in E} a_e (1 - m_e m'_e) \cdot
\left( 1 + \sum_{e \in E} \left( \frac{\cof(D_e)}{\kappa(D_e)} - 1
\right) \right),
\]
where the denominators are again placeholders.

\item For all such trees and all nodes $v_0 \in V$,
$\kappa(\dt) = \prod_{e \in E} a_e (1 - m_e m'_e)$. In particular,
$\kappa(\cdot)$ is multiplicative across edges of trees.
\end{enumerate}
\end{prop}

Next, from the above three results we deduce a few consequences. First,
the Graham--Hoffman--Hosoya type formulas proved in this section hold in
slightly greater generality:

\begin{cor}
Notation as in Definition~\ref{Dthird}. Let $G_1, \dots, G_k$ be strongly
connected subgraphs of $G$ containing $v_0 \in  V$ such that the sets
$V(G_j) \setminus \{ v_0 \}$ are pairwise disjoint. Also attach finitely
many pendant trees $\mathcal{T}_1, \dots, \mathcal{T}_l$ to $v_0$. Then
the formulas in Theorem~\ref{Tghh3} extend to this setting:
\begin{align}
\begin{aligned}
\kappa(D_G, v_0) = &\ \prod_{j=1}^k \kappa(D_{G_j}, v_0)
\prod_{i=1}^l \kappa(D_{\mathcal{T}_i}),\\
\frac{\det(D_G)}{\kappa(D_G, v_0)} = &\ 
\sum_{j=1}^k \frac{\det(D_{G_j})}{\kappa(D_{G_j}, v_0)}
+ \sum_{i=1}^l
\frac{\det(D_{\mathcal{T}_i})}{\kappa(D_{\mathcal{T}_i})},\\
\frac{\cof(D_G)}{\kappa(D_G, v_0)} - 1 = &\ 
\sum_{j=1}^k \left( \frac{\cof(D_{G_j})}{\kappa(D_{G_j}, v_0)} - 1
\right) + \sum_{i=1}^l \left(
\frac{\cof(D_{\mathcal{T}_i})}{\kappa(D_{\mathcal{T}_i})} - 1 \right),
\end{aligned}
\end{align}
where the denominators on the right are placeholders as earlier, and get
cancelled upon multiplying by the denominators on the left.
\end{cor}

\noindent We skip the proof as this result is a straightforward
consequence of Theorems~\ref{Tghh3} and~\ref{Tthird}.

Second, when given a graph $G$ with a usual, additive distance matrix
$(d(i,j))_{i,j \in V}$, recall that one can treat $D_G$ as the $q=1$
specialization of the matrix
$D_q(G) := \frac{1}{q-1} (q^{d(i,j)} - 1)_{i,j \in V}$.
We now claim $\kappa(D_G, v_0) \to \cof(D_G)$ when $q \to 1$. Indeed,
beginning with Theorem~\ref{Tthird}, we obtain:

\begin{cor}
The formulas in Theorems~\ref{Tghh3} and~\ref{Tthird} with $m(i,j) =
q^{d(i,j)}$ specialize as $q \to 1$ to the classical
Graham--Hoffman--Hosoya formulas~\eqref{Eghh}.
\end{cor}

\begin{proof}
This is easy to see in the setting of Theorem~\ref{Tthird} (and
Proposition~\ref{Pthird}) for trees, using above results with $a_e =
1/(q-1)\ \forall e \in E$. For general graphs, begin with the formula
\[
\kappa(D_G, v_0)|_{q \to 1} := \lim_{q \to 1} \det \left( [D]_q|_{V
\setminus \{ v_0 \}} - [{\bf u}_1]_q \, {\bf e}^T - [{\bf m}(V \setminus
\{ v_0 \}, v_0)]_q [{\bf w}_1]_q^T \right),
\]
where $[{\bf u}_1]_q$ is the vector with $j$th coordinate $(q^{u_j} -
1)/(q-1)$, etc. Now by the polynomiality of the determinant in its
entries, setting $q=1$ on the right-hand side after taking the
determinant is the same as setting it before; and in the latter scenario,
we have
\[
\kappa(D_G, v_0)|_{q \to 1} = \det (D_{V \setminus \{ v_0 \}} - {\bf
u}_1|_{q \to 1} {\bf e}^T - {\bf e} \, {\bf w}_1^T|_{q \to 1}).
\]
But this is precisely $\cof(D_G)$, as observed in~\cite{GHH}. We are done
by Theorem~\ref{Tghh3}.
\end{proof}

Finally, we present an example which shows that the
Graham--Hoffman--Hosoya type identities in Theorem~\ref{Tghh3} do not
uniformly hold in greater generality.

\begin{example}\label{Exthird}
While $\det(\dt), \cof(\dt)$ for trees $\mathcal{T}$ depend only on the
strong blocks -- i.e.~edges -- of $\mathcal{T}$, the same does not hold
for general graphs. For example, let $G$ consist of one cut-vertex $v_0$
and two strong blocks: an edge $e$ with data $(a_e, m_e = m'_e)$;
and the clique $K_3$, with distance matrix
\[
D_{K_3} = \begin{pmatrix}
0 & a(m - 1) & c(q - 1) \\
a(m - 1) & 0 & b(n - 1) \\
c(q - 1) & b(n - 1) & 0
\end{pmatrix}.
\]
We now claim that the quantity
\[
\kappa(D_G, v_0) = \kappa(D_{K_3}, v_0) \cdot a_e (1 - m_e^2)
\]
(from above results) does depend on the location of the cut-vertex $v_0$.
Indeed, to show this it suffices to verify that $\kappa(D_{K_3}, 1) \neq
\kappa(D_{K_3}, 2) \neq \kappa(D_{K_3}, 3)$. But an easy computation
shows that
\[
\kappa(D_{K_3}, 3) = \det \begin{pmatrix}
c(1 - q^2) & a(m-1) - c(q-1) - qb(n-1)\\
a(m-1) - b(n-1) - nc(q-1) & b(1 - n^2)
\end{pmatrix}.
\]
Now $a^2 m^2, b^2 n^2$ have coefficients $1,q$ in $\kappa(D_{K_3}, 3)$,
respectively. As this is not ``symmetric'', it follows that
$\kappa(D_{K_3}, 3) \neq \kappa(D_{K_3},1)$; the remaining verifications
are similar. \qed
\end{example}

\subsection{Proofs}

\begin{proof}[Proof of Theorem~\ref{Tghh3}]
It suffices by induction on $k$ to prove the result when $k=2$. Thus,
suppose
\[
G = G_1 \sqcup_{v_0} G_2, \quad \text{with} \quad V(G_1) = \{ 1, \dots,
v_0 \}, \ V(G_2) = \{ v_0, \dots, n \}.
\]
Let $V'_j := V(G_j) \setminus \{ v_0 \}$ for $j=1,2$. Corresponding to
this notation, write the distance matrix as
\[
D_G := \begin{pmatrix}
D_1 & {\bf u}_1 & {\bf u}_1 {\bf e}^T + {\bf m}(V'_1, v_0) {\bf w}_2^T \\
{\bf w}_1^T & 0 & {\bf w}_2^T \\
{\bf u}_2 {\bf e}^T + {\bf m}(V'_2, v_0) {\bf w}_1^T & {\bf u}_2 & D_2
\end{pmatrix}
\]
in block form. Let
\begin{equation}\label{Edprime}
D'_j := D_j - {\bf u}_j {\bf e}^T - {\bf m}(V'_j, v_0) {\bf w}_j^T,
\qquad j=1,2.
\end{equation}

The first claim is that computing $\kappa(D_G, v_0)$ yields precisely the
determinant of the block-diagonal matrix $\begin{pmatrix} D'_1 & 0 \\ 0 &
D'_2 \end{pmatrix}$. This is straightforward, and it follows that
\[
\kappa(D_G, v_0) = \det(D'_1) \det(D'_2) = 
\kappa(D_{G_1}, v_0) \kappa(D_{G_2}, v_0).
\]

We next show the identity for $\det(D_G)$. Begin with the block matrix
$D_G$ as above, and carry out the sequence of block row operations
\[
R_1 \mapsto R_1 - {\bf m}(V'_1, v_0) \, R_2, \qquad
R_3 \mapsto R_3 - {\bf m}(V'_2, v_0) \, R_2,
\]
followed by the sequence of block-column operations
\[
C_1 \mapsto C_1 - C_2 \, {\bf e}^T, \qquad
C_3 \mapsto C_3 - C_2 \, {\bf e}^T.
\]
This yields precisely the matrix
\[
D' := \begin{pmatrix}
D'_1 & {\bf u}_1 & 0 \\
{\bf w}_1^T & 0 & {\bf w}_2^T \\
0 & {\bf u}_2 & D'_2
\end{pmatrix},
\]
where $D'_1, D'_2$ were defined in~\eqref{Edprime}. Now we may assume by
a Zariski density argument that (in our setting) $D'_1, D'_2$ are
invertible, since they are nonzero matrices with no constraints. Carrying
out block row and column operations on $R_2, C_2$ yields a
block-triangular matrix, and we obtain:
\begin{align*}
\det D_G = \det D' = &\ \det(D'_1) \det (D'_2) \cdot (- {\bf w}_1^T
(D'_1)^{-1} {\bf u}_1 - {\bf w}_2^T (D'_2)^{-1} {\bf u}_2 )\\
= &\ \det(D'_2) \det \begin{pmatrix} D'_1 & {\bf u}_1 \\ {\bf w}_1^T & 0
\end{pmatrix} + 
\det(D'_1) \det \begin{pmatrix} 0 & {\bf w}_2^T \\ {\bf u}_2 &  D'_2
\end{pmatrix},
\end{align*}
where the final equality uses two Schur complement expansions of
determinants. But now in the two block $2 \times 2$ matrices in the final
expression, we may replace $D'_j$ by $D_j$ by performing block row and
column operations. Hence,
\[
\det(D_G) = \det (D'_2) \det (D_{G_1}) + \det (D'_1) \det (D_{G_2}).
\]
This shows (by Zariski density) the formula for $\det(D_G)$,
since $\det (D'_j) = \kappa(D_{G_j}, v_0)$ for $j=1,2$.

Finally, we show the claimed identity for $\cof(D_G)$. Begin with the
matrix $D_G + xJ$, where $D_G$ is as above. Carrying out the block-column
operations
\[
C_1 \mapsto C_1 - C_2 \, {\bf e}^T, \qquad
C_3 \mapsto C_3 - C_2 \, {\bf e}^T.
\]
on $D_G+xJ$, and denoting ${\bf m}_j := {\bf m}(V'_j, v_0)$ for
convenience, we compute:
\[
\det(D_G) + x \cof(D_G) = \det \begin{pmatrix}
D_1 - {\bf u}_1 {\bf e}^T & {\bf u}_1 + x {\bf e} & {\bf m}_1 {\bf w}_2^T
\\
{\bf w}_1^T & x & {\bf w}_2^T \\
{\bf m}_2 {\bf w}_1^T & {\bf u}_2 + x {\bf e} & D_2 - {\bf u}_2 {\bf e}^T
\end{pmatrix}.
\]
By linearity of $\det(\cdot)$ in the second row (treated as a polynomial
in $x$ with vector coefficients), and taking the linear term in $x$, we
obtain via Lemma~\ref{Ldetcof}:
\[
\cof(D_G) = \det \begin{pmatrix}
D_1 - {\bf u}_1 {\bf e}^T & {\bf e} & {\bf m}_1 {\bf w}_2^T \\
{\bf w}_1^T & 1 & {\bf w}_2^T \\
{\bf m}_2 {\bf w}_1^T & {\bf e} & D_2 - {\bf u}_2 {\bf e}^T
\end{pmatrix}
= \det \begin{pmatrix}
D'_1 & {\bf e} - {\bf m}_1 & 0 \\
{\bf w}_1^T & 1 & {\bf w}_2^T \\
0 & {\bf e} - {\bf m}_2 & D'_2
\end{pmatrix},
\]
where $D'_1, D'_2$ are as in~\eqref{Edprime}, and the final equality uses
two block-row operations. Again assuming $D'_1, D'_2$ are invertible over
$R$ by Zariski density, we obtain via block-row operations:
\begin{equation}\label{Ecof-kappa}
\cof(D_G) = \kappa(D_{G_1}, v_0) \kappa(D_{G_2}, v_0) (1
- {\bf w}_1^T (D'_1)^{-1} ({\bf e} - {\bf m}_1)
- {\bf w}_2^T (D'_2)^{-1} ({\bf e} - {\bf m}_2)),
\end{equation}
where $\kappa(D_{G_j}, v_0) = \det(D'_j)$ s above.

A similar analysis for the matrix $D_{G_j}$ for $j=1,2$ reveals that
\[
\cof(D_{G_j}) = \det(D_j - {\bf e} {\bf w}_j^T - {\bf u}_j {\bf e}^T).
\]
But this is the determinant of a rank-one update of $D'_j$, so using
Schur complements,
\[
\cof(D_{G_j}) = \det(D'_j - ({\bf e} - {\bf m}_j) {\bf w}_j^T) = \det
\begin{pmatrix} D'_j & {\bf e} - {\bf m}_j \\ {\bf w}_j^T & 1
\end{pmatrix}
= \kappa(D_{G_j}, v_0) (1 - {\bf w}_j^T (D'_j)^{-1} ({\bf e} - {\bf
m}_j)).
\]
Combining this with~\eqref{Ecof-kappa}, the result follows.
\end{proof}

\begin{proof}[Proof of Theorem~\ref{Tthird}]
Consider a tree with edge-data $\mathcal{T}$ and a node $v_0$. We prove
the result by induction on $|V|$, with the $|V|=2$ case (of a single
edge) easily verified. For the induction step, if $v_0$ is a cut-vertex
then we are done by Theorem~\ref{Tghh3}. Thus, suppose $v_0$ is a pendant
node, say $v_0 = n \geq 3$ and $n \sim p(n) = n-1$. Denoting
\[
a_{e_0} := a_{n-1,n}, \quad
m_{e_0} := m_{n-1,n}, \quad
m'_{e_0} := m_{n,n-1}, \quad
{\bf d} := {\bf d}([n-2], n-1), \quad
{\bf d}' := {\bf d}(n-1, [n-2])
\]
for notational convenience, the matrix $\dt$ is of the form
\[
\begin{pmatrix}
D|_{[n-2]} & {\bf d} & {\bf d} + a_{e_0} (m_{e_0} - 1) {\bf m}([n-2],n-1)
\\
({\bf d}')^T & 0 & a_{e_0} (m_{e_0} - 1)\\
a_{e_0} (m'_{e_0} - 1) {\bf e}^T + m'_{e_0} ({\bf d}')^T & a_{e_0}
(m'_{e_0} - 1) & 0
\end{pmatrix}.
\]

Writing $\dt = \begin{pmatrix} D|_{[n-1]} & {\bf u}_1 \\ {\bf w}_1^T & 0
\end{pmatrix}$, we compute:
\[
D_\kappa = \begin{pmatrix} D_0 & a_{e_0} (1 - m_{e_0} m'_{e_0}) {\bf
m}([n-2], n-1) \\ (1 - m_{e_0} m'_{e_0}) ({\bf d}')^T + a_{e_0} (1 -
m_{e_0} m'_{e_0}) {\bf e}^T & a_{e_0} (1 - m_{e_0} m'_{e_0})
\end{pmatrix},
\]
where
\[
D_0 = D|_{[n-2]} - {\bf d} \, {\bf e}^T
+ a_{e_0} (1 - m_{e_0} m'_{e_0}) {\bf m}([n-2], n-1) {\bf e}^T
- m_{e_0} m'_{e_0} {\bf m}([n-2], n-1) \, ({\bf d}')^T.
\]

Carry out the block-row operation $R_1 \mapsto R_1 - {\bf m}([n-2], n-1)
R_2$, to obtain a block lower triangular determinant:
\begin{align*}
\kappa(\dt, n) = &\ \det D_\kappa = \det \begin{pmatrix}
D|_{[n-2]} - {\bf d} \, {\bf e}^T - {\bf m}([n-2], n-1) \, ({\bf d}')^T
& 0\\
(1 - m_{e_0} m'_{e_0}) (a_{e_0} {\bf e}^T - ({\bf d}')^T) & a_{e_0} (1 -
m_{e_0} m'_{e_0}).
\end{pmatrix}
\end{align*}
But the $(1,1)$ block on the right-hand side has determinant precisely
$\kappa(\dt|_{[n-1]}, n-1)$. We are now done by the induction hypothesis.
Finally,~\eqref{Eqthird} is now straightforward from
Theorem~\ref{Tmaster}.
\end{proof}

\begin{proof}[Proof of Proposition~\ref{Pthird}]
We first assume that~(1) holds, and show~(3). Consider a tree with
edge-data $\mathcal{T}$ and a node $v_0$. Attach a pendant edge $e_0$ to
$v_0$ with edge-data $(a_{e_0}, m_{e_0}, m'_{e_0})$, and call the
resulting edge-data $\mathcal{T}_0$ -- note that $v_0$ is a cut-vertex in
$\mathcal{T}_0$ satisfying the assumptions in Definition~\ref{Dthird}.
Since it is easily verified that $\kappa(D_{e_0}, v_0) = a_{e_0} (1 -
m_{e_0} m'_{e_0})$, Theorem~\ref{Tghh3} yields:
\[
\det(D_{\mathcal{T}_0}) = \kappa(\dt, v_0) \det(D_{e_0}) +
\kappa(D_{e_0}, v_0) \det(\dt).
\]
But in this equation, all terms except $\kappa(\dt, v_0)$ are known
by~(1) and direct computation. From this, and a Zariski density argument
that allows one to cancel $\det(D_{e_0})$, the assertion~(3) follows.

A similar argument shows $(2) \implies (3)$ -- now assuming by Zariski
density that $(\cof - \kappa)(D_e)$ and $\kappa(D_e)$ are invertible for
each edge $e$.
Conversely, suppose~(3) holds. Then~(1) and~(2) are direct consequences
of Theorem~\ref{Tghh3} and direct computation for $\det(D_e)$.
\end{proof}
%}}}

\subsection*{Acknowledgments}

The authors are grateful to the referee(s) for carefully going
through the manuscript in detail, and for their helpful comments which
improved the paper.
P.N.C.\ was partially supported by
INSPIRE Faculty Fellowship research grant DST/INSPIRE/04/2021/002620
(DST, Govt.~of India),
IIT Gandhinagar Internal Project grant IP/IITGN/MATH/PNC/2223/25,
C.V.\ Raman Postdoctoral Fellowship 80008664 (IISc),
National Post-Doctoral Fellowship PDF/2019/000275 (SERB, Govt.\ of
India), and
NBHM--DAE Postdoctoral Fellowship 0204/11/2018/R\&D-II/6437.
A.K.\ was partially supported by
Ramanujan Fellowship SB/S2/RJN-121/2017,
MATRICS grant MTR/2017/000295, and
SwarnaJayanti Fellowship grants SB/SJF/2019-20/14 and DST/SJF/MS/2019/3
from SERB and DST (Govt.~of India),
grant F.510/25/CAS-II/2018(SAP-I) from UGC (Govt.~of India),
and by a Young Investigator Award from the Infosys Foundation.

%{{{1 Bibliography

%\bibliographystyle{plain}
%\bibliography{biblio-AIM}

%}}}

\end{document}